\documentclass[11pt]{article}
\usepackage{amscd, amsmath, amssymb, amsthm}
\usepackage[all,cmtip]{xy}

\title{Pseudo-abelian varieties}
\author{Burt Totaro}
\date{  }

\def\Z{\text{\bf Z}}

\def\P{\text{\bf P}}
\def\A{\text{\bf A}}
\def\F{\text{\bf F}}
\def\arrow{\rightarrow}

\def\imp{\Rightarrow}
\def\surj{\twoheadrightarrow}

\def\fppf{\text{fppf}}
\def\Spec{\text{Spec}}
\def\Ext{\text{Ext}}
\def\Hom{\text{Hom}}

\def\Pic{\text{Pic}}

\def\Br{\text{Br}}
\def\perf{\text{perf}}
\def\deg{\text{deg}}
\def\Ga{\mathbf{G}_a}
\def\Gm{\mathbf{G}_m}
\def\Gal{\text{Gal}}
\def\tors{\text{tors}}
\def\m{\mathfrak{m}}
\def\Proj{\text{Proj}\, }

\begin{document}
\maketitle

\newtheorem{theorem}{Theorem}[section]
\newtheorem{proposition}[theorem]{Proposition}
\newtheorem{corollary}[theorem]{Corollary}
\newtheorem{lemma}[theorem]{Lemma}

\theoremstyle{definition}
\newtheorem{definition}[theorem]{Definition}
\newtheorem{example}[theorem]{Example}
\newtheorem{question}[theorem]{Question}

\theoremstyle{remark}
\newtheorem{remark}[theorem]{Remark}

The theory of algebraic groups is divided into two parts
with very different flavors:
affine algebraic groups (which can be viewed as matrix groups) and abelian
varieties. Concentrating on these two types of groups makes sense
in view of Chevalley's theorem: for a perfect field $k$,
every smooth connected $k$-group $G$
is an extension of an abelian variety $A$ by a smooth connected
affine $k$-group $N$
\cite{Chevalley, Conradchev}:
$$1\arrow N\arrow G\arrow A\arrow 1.$$
But Chevalley's theorem
fails over every imperfect field. What can be said about the structure
of a smooth connected algebraic group
over an arbitrary field $k$? (Group schemes which are
neither affine nor proper come up naturally,
for example as the automorphism group scheme or the Picard scheme
of a projective variety over $k$. Groups over an imperfect
field such as the rational function field $\F_p(t)$ arise
geometrically as the generic fiber of a family of groups
in characteristic $p$.)

One substitute for Chevalley's theorem that works over an arbitrary field
is that every connected group scheme (always assumed to be of finite
type) over a field $k$ is an extension of an abelian variety
by a connected affine group scheme, not uniquely \cite[Lemme IX.2.7]{Raynaud}.
But when this result is applied to a smooth $k$-group,
the affine subgroup scheme may have to be non-smooth.
And it is desirable to understand the
structure of smooth $k$-groups as far as possible without bringing
in the complexities of arbitrary $k$-group schemes. To see how
far group schemes can be from being smooth, note that
every group scheme $G$ of finite type
over a field $k$ has a unique maximal smooth
closed $k$-subgroup \cite[Lemma C.4.1]{CGP},
but (for $k$ imperfect) that subgroup
can be trivial even when $G$ has positive dimension.
(A simple example is the group scheme $G=\{(x,y)\in (\Ga)^2: x^p=ty^p \}$
for $t\in k$ not a $p$th power, where $p$ is the characteristic
of $k$. The dimension of $G$ is 1, but the maximal
smooth $k$-subgroup of $G$ is the trivial group.)

Brion gave
a useful structure theorem
for smooth $k$-groups
by putting the smooth affine group ``on top''.
Namely,
for any field $k$ of positive characteristic,
every
smooth connected
$k$-group is a central extension of a smooth connected affine $k$-group
by a semi-abelian variety (an extension of an abelian variety
by a torus) \cite[Proposition 2.2]{Brion}.
(Another proof was given
by C.~Sancho de Salas and F.~Sancho de Salas \cite{SS}.)
One can still ask what substitute
for Chevalley's theorem works over arbitrary fields, with the
smooth affine group ``on the bottom''. We can gain inspiration from Tits's
theory of pseudo-reductive groups \cite{Tits91, Tits92},
developed by Conrad-Gabber-Prasad
\cite{CGP}. By definition, a pseudo-reductive
group over a field $k$ is a smooth connected affine $k$-group $G$ 
such that every smooth connected unipotent normal
$k$-subgroup of $G$ is trivial. That suggests the definition:

\begin{definition}
A {\it pseudo-abelian variety }over a field $k$
is a smooth connected $k$-group $G$
such that every smooth connected affine normal $k$-subgroup of $G$
is trivial.
\end{definition}

It is immediate that every smooth connected group 
over a field $k$ is an extension of
a pseudo-abelian variety by a smooth connected
affine group over $k$, in a unique way. Whether this is useful
depends on what can be said about the structure of pseudo-abelian
varieties. Chevalley's theorem implies that a pseudo-abelian
variety over a perfect field is simply an abelian variety.

Over any imperfect field, Raynaud constructed pseudo-abelian varieties
which are not abelian varieties \cite[Exp.~XVII, App.~III, Prop.~5.1]{SGA3}.
Namely, for any finite purely inseparable
extension $l/k$ and any abelian variety $B$ over $l$, the Weil
restriction $R_{l/k}B$ is a pseudo-abelian variety, and it
is not an abelian variety if $l\neq k$ and $B\neq 0$. (Weil restriction
produces a $k$-scheme $R_{l/k}B$ whose set of $k$-rational points is
equal to the set of $l$-rational points of $B$.)
Indeed, over an algebraic
closure $\overline{k}$ of $l$,
$R_{l/k}B$ becomes an extension of $B_{\overline{k}}$
by a smooth unipotent group of dimension $([l\colon k]-1)\dim(B)$, and so
$R_{l/k}B$ is not an abelian variety. (This example shows that
the notion of a pseudo-abelian variety is not geometric, in the
sense that it is not preserved by arbitrary field extensions.
It is preserved by separable field extensions, however.)

One main result of this paper is that every pseudo-abelian variety
over a field $k$ is {\it commutative}, and every
pseudo-abelian variety is an extension of a smooth connected
commutative unipotent $k$-group by an abelian variety (Theorem
\ref{pseudo}). In this sense,
pseudo-abelian varieties are reasonably close to abelian varieties.
So it is a meaningful generalization of Chevalley's theorem
to say that every smooth connected group over a field $k$ is an extension
of a pseudo-abelian variety by a smooth connected affine group
over $k$.

One can expect many properties of abelian varieties to extend
to pseudo-abelian varieties. For example, the Mordell-Weil
theorem holds for pseudo-abelian varieties (Proposition \ref{mordell}).
Like abelian varieties, pseudo-abelian varieties
can be characterized among all smooth connected
groups $G$ over a field $k$ without using the group structure, in fact
using only the birational equivalence class of $G$ over $k$:
$G$ is a pseudo-abelian variety if and only if $G$ is not ``smoothly
uniruled'' (Theorem \ref{bir}).

The other main result is that,
over an imperfect field of characteristic $p$,
{\it every }smooth connected commutative
group of exponent $p$
occurs as the unipotent quotient of some pseudo-abelian
variety (Corollaries \ref{supersingular} and \ref{ordinary}).
Over an imperfect field,
smooth commutative unipotent groups form a rich family,
studied by Serre, Tits, Oesterl\'e, and others
over the past 50 years \cite{KMT}, \cite{Oesterle}, \cite[Appendix B]{CGP}.
So there are far more pseudo-abelian varieties (over any imperfect field)
than the initial examples, Weil restrictions of abelian
varieties.

Lemma \ref{equiv} gives a precise relation between
the structure of certain
pseudo-abelian varieties and the (largely unknown) structure
of commutative pseudo-reductive groups.
We prove some new results about commutative
pseudo-reductive groups.
First, a smooth connected unipotent group of dimension 1 over a
field $k$ occurs as the unipotent quotient of some 
commutative pseudo-reductive group if and only if it is not
isomorphic to the additive group $\Ga$ over $k$ (Corollary \ref{cordim1}).
But an analogous statement fails in dimension 2 (Example \ref{dim2}).
The proofs include some tools for computing the invariants $\Ext^1(U,\Gm)$
and $\Pic(U)$ of a unipotent group $U$. Finally, Question \ref{ext2}
conjectures a calculation
of $\Ext^2(\Ga,\Gm)$ over any field 
by generators and relations, in the spirit
of the Milnor conjecture. Question \ref{dimone} attempts
to describe the commutative pseudo-reductive groups
over 1-dimensional fields.

Thanks to Lawrence Breen, Michel Brion, Brian Conrad, and Tony Scholl
for useful discussions. The proofs of Theorem \ref{pseudo}
and Lemma \ref{alphap} were simplified by Brion and Conrad,
respectively. Other improvements are due to the excellent referees,
including Example \ref{UU}, which answers a question in an earlier
version of the paper.

\section{Notation}
\label{notation}

A {\it variety }over a field $k$ means an integral separated
scheme of finite type over $k$.
Let $k$ be a field with algebraic closure $\overline{k}$
and separable closure $k_s$.
A field extension $F$ of $k$ (not necessarily algebraic) is {\it separable }if
the ring $F\otimes_k\overline{k}$ contains no nilpotent elements
other than zero. For example, the function field of a variety $X$
over $k$ is separable over $k$ if and only if the smooth locus
of $X$ over $k$ is nonempty \cite[section X.7, Theorem 1, Remark 2,
Corollary 2]{Bourbaki}. 

We use the convention that
a connected topological space is nonempty.

A group scheme over a field $k$ is {\it unipotent }if it is isomorphic
to a $k$-subgroup scheme of the group of strictly upper triangular
matrices in $GL(n)$ for some $n$ 
(see \cite[Th\'eor\`eme XVII.3.5]{SGA3} for several equivalent
conditions). Being unipotent is a geometric property,
meaning that it does not change under field extensions of $k$. Unipotence
passes to subgroup schemes, quotient groups, and group extensions.

We write $\Ga$ for the additive group.
Over a field $k$ of characteristic $p>0$, we write
$\alpha_p$ for the $k$-group scheme $\{ x\in \Ga: x^p=0\}$.
A group scheme over $k$ is unipotent if and only if it has a composition
series with successive quotients isomorphic to $\alpha_p$, $\Ga$,
or $k$-forms of $(\Z/p)^r$ \cite[Th\'eor\`eme XVII.3.5]{SGA3}.

Tits defined a smooth connected unipotent group over a field $k$ to
be {\it $k$-wound }if it does not contain $\Ga$ as a subgroup
over $k$. When $k$ has
characteristic $p>0$, a smooth connected
commutative $k$-group of exponent $p$
can be described in a unique way as an extension of a $k$-wound
group by a subgroup isomorphic to $(\Ga)^n$ for some $n\geq 0$
\cite[Theorem B.3.4]{CGP}.
Over a perfect field,
a $k$-wound group is trivial. An example of a nontrivial $k$-wound group
is the smooth connected subgroup $\{(x,y): y^p=x-tx^p\}$ of $(\Ga)^2$ for any
$t\in k-k^p$, discussed in Example \ref{dim1ex}.

Over an imperfect field $k$ of characteristic $p$, there are many
smooth connected commutative groups of
exponent $p$ (although
they all become isomorphic to $(\Ga)^n$ over the algebraic closure
of $k$). One striking phenomenon is that some of these groups
are $k$-rational varieties, while others contain no $k$-rational curves
\cite[Theorem 6.9.2]{KMT}, \cite[Theorem VI.3.1]{Oesterle}.
Explicitly, define a {\it $p$-polynomial }
to be a polynomial with coefficients in $k$
such that every monomial in $f$ is a single 
variable raised to some power of $p$. Then
every smooth connected commutative $k$-group
of exponent $p$ and dimension $n$
is isomorphic to the subgroup of $(\Ga)^{n+1}$ defined by some
$p$-polynomial $f$ with nonzero degree-1 part
\cite[Proposition V.4.1]{Oesterle}, \cite[Proposition B.1.13]{CGP}.

A smooth connected affine group $G$ 
over a field $k$ is {\it pseudo-reductive }if
every smooth connected unipotent normal $k$-subgroup of $G$ is trivial.
The stronger property that $G$ is {\it reductive }means that every smooth
connected unipotent normal subgroup of $G_{\overline{k}}$ is trivial.

We write $\Gm$ for the multiplicative group over $k$.
For each positive integer
$n$, the $k$-group scheme $\{x\in \Gm: x^n=1\}$ of $n$th roots of unity
is called $\mu_n$. A $k$-group scheme $M$ is of {\it multiplicative type }if
it is the dual of some $\Gal(k_s/k)$-module $L$ which is finitely
generated as an abelian group, meaning that $M=\Spec (k_s[L])^{\Gal(k_s/k)}$
\cite[Proposition X.1.4]{SGA3}. Dualizing the surjection $L\arrow L/L_{\tors}$
shows that every $k$-group scheme $M$
of multiplicative type contains a $k$-torus $T$ 
with $M/T$ finite. (Explicitly, $T$ is the identity
component of $M$ with reduced scheme structure.)

\section{Structure of pseudo-abelian varieties}

\begin{theorem}
\label{pseudo}
Every pseudo-abelian variety $E$ over a field $k$ is commutative.
Moreover, $E$ is in a unique way
an extension of a smooth connected commutative unipotent $k$-group
$U$ by an abelian variety $A$:
$$1\arrow A\arrow E\arrow U\arrow 1$$
Finally, $E$ can be written (not uniquely)
as $(A\times H)/K$
for some commutative affine $k$-group scheme $H$
and some commutative finite $k$-group scheme $K$ which injects
into both $A$ and $H$, with $H/K\cong U$.
\end{theorem}

\begin{proof}
Since $E$ is a smooth connected $k$-group, the commutator
subgroup $[E,E]$ is a smooth connected normal $k$-subgroup of $E$
\cite[Proposition VIB.7.1]{SGA3}. Since abelian varieties
are commutative, Chevalley's theorem
applied to $E_{\overline{k}}$
gives that $[E,E]_{\overline{k}}$ is affine
\cite{Chevalley, Conradchev}.
Therefore the $k$-subgroup $[E,E]$ is affine.
Since $E$ is a pseudo-abelian variety over $k$, it follows
that $[E,E]$ is trivial. That is, $E$ is commutative.

If the field $k$
is perfect, then the pseudo-abelian variety
$E$ is an abelian variety by Chevalley's theorem.
So we can assume that $k$ is imperfect; in particular,
$k$ has characteristic $p>0$.
By Brion's theorem,
$E$ is an extension
$$1\arrow A\arrow E\arrow U\arrow 1$$
with $A$ a semi-abelian variety and $U$ a smooth connected
affine $k$-group \cite[Proposition 2.2]{Brion}.
The maximal $k$-torus in $A$ is trivial
because $E$ is a pseudo-abelian variety.
That is, $A$ is an abelian variety.
So the morphism $E\arrow U$ is proper and flat, with geometrically
reduced and connected fibers. It follows that
the pullback map $O(U)\arrow O(E)$ on rings of regular
functions is an isomorphism
\cite[Proposition 7.8.6]{EGAIII2}. Since $U$ is affine,
it follows that $U=\Spec \, O(E)$ and hence
the exact sequence is uniquely
determined by $E$.
(The idea of considering $\Spec \, O(E)$ goes back to Rosenlicht
\cite[p.~432]{Rosenlichtbasic}.)

Like any connected group scheme of finite type over $k$,
$E$ can also be written (not uniquely) as an extension
$$1\arrow H\arrow E\arrow B\arrow 1$$
with $H$ a connected affine group scheme over $k$ and $B$ an abelian variety
\cite[Lemme IX.2.7]{Raynaud}.
Let $K$ be the intersection of $H$ and $A$ in $E$.
Then $K$ is both affine and proper over $k$,
and so $K$ has dimension 0.
Also, $H/K$ injects into $U$, and the abelian variety $B$
maps onto the quotient group $U/(H/K)$. Since $U/(H/K)$ is both
affine (being a quotient group of $U$) and an abelian variety,
it is trivial. That is, $H/K$ maps isomorphically to $U$.
Since $E$ is commutative, this means that
$E$ is isomorphic to $(A\times H)/K$.

It remains to show that $U$ is unipotent. 
Since $H$ is a commutative affine $k$-group scheme, it is an extension
of a unipotent $k$-group scheme by a $k$-group scheme $M$ of multiplicative
type \cite[Th\'eor\`eme XVII.7.2.1]{SGA3}.
Because $M\subset H \subset E$ where $E$
is a pseudo-abelian variety, every $k$-torus in $M$ is trivial.
By section \ref{notation}, it follows
that $M$ is finite. Thus $H$ is an extension of a unipotent $k$-group
scheme by a finite $k$-group scheme. So the quotient group $U$ of $H$
is also an extension of a unipotent $k$-group scheme by a finite
$k$-group scheme; in particular, every $k$-torus
in $U$ is trivial. Since $U$ is a smooth
connected affine $k$-group, it follows that $U$ is unipotent
\cite[Proposition XVII.4.1.1]{SGA3}.
\end{proof}

\begin{question}
\label{brionquestion}
(Suggested by Michel Brion.) How can
Raynaud's examples of pseudo-abelian
varieties, purely inseparable Weil restrictions of abelian varieties,
be described explicitly as extensions $1\arrow A\arrow E\arrow U\arrow 1$
or as quotients $(A\times H)/K$, in the terminology of
Theorem \ref{pseudo}?
\end{question}

For a finite purely inseparable
extension $l/k$ and an abelian variety $B$ over $l$, the maximal
abelian subvariety of the Weil restriction $R_{l/k}B$
is the Chow $l/k$-trace
of $B$ \cite{Conradtrace}. Question \ref{brionquestion} asks for
a description of the unipotent quotient of $R_{l/k}B$, too.

\begin{lemma}
Let $G$ be a smooth connected group over a field $k$,
and let $K$ be a separable extension field of $k$.
Then $G$ is a pseudo-abelian variety over $k$
if and only if it becomes a pseudo-abelian variety
over $K$.
\end{lemma}

\begin{proof}
If $G$ becomes a pseudo-abelian variety over $K$, it is clearly
a pseudo-abelian variety over $k$. For the converse, 
by considering the separable closure of $K$, it suffices to treat
the cases where (1) $K$ is the separable closure of $k$
or (2) $k$ is separably closed. To prove (1): 
there is a unique maximal smooth connected affine normal $k_s$-subgroup
of $G_{k_s}$. By uniqueness, it is $\Gal(k_s/k)$-invariant, and therefore
comes from a subgroup $H$ over $k$. Clearly $H$ is 
a smooth connected affine normal $k$-subgroup of $G$.
To prove (2), reduce to the case where
$K$ is finitely generated over $k$,
so that $K$ is the fraction field of a smooth $k$-variety $X$, shrink
$X$ so that the maximal smooth connected affine normal $K$-subgroup
of $G_K$ comes from a subgroup scheme of $G_X$, and specialize
to a $k$-point of $X$ (which exists because $k$ is separably closed).
This is essentially the same as the proof
that pseudo-reductivity
remains unchanged under separable extensions \cite[Proposition 1.1.9(1)]{CGP}.
\end{proof}

\section{Example}

Pseudo-abelian varieties
occur in nature, in the following sense.

\begin{example}
For every odd prime $p$, there is a regular projective curve $X$
over a field $k$ of characteristic $p$
such that the Jacobian $\Pic^0_{X/k}$ is a pseudo-abelian
variety which is not an abelian variety.
\end{example}

We leave it to the reader to seek a curve with these
properties in characteristic 2. (The simpler the example, the better.)

\begin{proof}
Let $k$ be the rational function field $\F_p(t)$.
Let $X$ be the regular compactification of the regular affine
curve $y^2=x(x-1)(x^p-t)$ over $k$. Rosenlicht considered this
curve for a closely related purpose \cite[pp.~49--50]{Rosenlichtsome}.
(To find the non-regular locus
of the given affine curve, compute the zero locus
of all derivatives of the equation with respect to $x,y$
and also $t$: this gives that $(2x-1)(x^p-t)=0$, $2y=0$,
and $x(x-1)=0$, which defines the empty set
in $\A^2_{k}=\A^2_{\F_p(t)}$.)
Then $X$ is a geometrically integral projective curve of arithmetic
genus $(p+1)/2$, and so $G:=\Pic^0_{X/k}=\ker(\deg\colon \Pic_{X/k}\arrow \Z)$
is a smooth connected commutative
$k$-group of dimension $(p+1)/2$
\cite[Theorem 8.2.3 and Proposition 8.4.2]{BLR}.
Over an algebraic closure
$\overline{k}$, the curve $X_{\overline{k}}$ is not regular:
it has a cusp (of the form $z^2=w^p$)
at the point $(x,y)=(u,0)$, where we define $u=t^{1/p}$ in $\overline{k}$.
The normalization $C$ of $X_{\overline{k}}$ is the regular
compactification of the regular affine curve $y^2=x(x-1)(x-u)$ over
$\overline{k}$, with normalization map $C\arrow X_{\overline{k}}$
given by $(x,y)\mapsto (x,y(x-u)^{(p-1)/2})$. Since $C$ has genus 1,
$\Pic^0_{C/\overline{k}}$ is an elliptic curve over $\overline{k}$.
Pulling back by $C\arrow X_{\overline{k}}$ gives a homomorphism
from $G_{\overline{k}}$ onto $\Pic^0_{C/\overline{k}}$,
$$1\arrow N\arrow G_{\overline{k}}\arrow \Pic^0_{C/\overline{k}}\arrow 1,$$
with kernel $N$ isomorphic to $(\Ga)^{(p-1)/2}$ over
$\overline{k}$ \cite[section V.17]{SerreAlgebraic},
\cite[Proposition 9.2.9]{BLR}. It follows that $G$ is not an
abelian variety over $k$.

To show that $G$ is a pseudo-abelian variety over $k$,
we have to show that every smooth connected affine $k$-subgroup $S$
of $G$ is trivial. For such a subgroup, $S_{\overline{k}}$ must map
trivially into the elliptic curve $\Pic^0_{C/\overline{k}}$. So it suffices
to show that every smooth connected $k$-subgroup $S$ of $G$ with
$S_{\overline{k}}$ contained in $N$ is trivial. It will be 
enough to prove the corresponding statement at the level of Lie algebras.
Namely, we have an exact sequence of $\overline{k}$-vector spaces
$$0\arrow N\arrow H^1(X,O)\otimes_k \overline{k}\arrow H^1(C,O)\arrow 0,$$
and it suffices to show that the codimension-1 $\overline{k}$-linear
subspace $N$
has zero intersection with the $k$-vector space $H^1(X,O)$.

The dual of the surjection $H^1(X,O)\otimes_k\overline{k}\arrow
H^1(C,O)$ is the inclusion 
$H^0(C,K_C)\arrow H^0(X,K_X)\otimes_k \overline{k}$ given
by the trace map associated to the finite birational
morphism $C\arrow X_{\overline{k}}$.
Here
$K_X$ denotes the canonical line bundle of the Gorenstein curve $X$.
It is a standard calculation for hyperelliptic curves that
$H^0(C,K_C)$ has a $\overline{k}$-basis given by $dx/y$ and $H^0(X,K_X)$
has a $k$-basis given by $x^i dx/y$ for $0\leq i\leq (p-1)/2$
\cite[section 2]{Stohr}.
By the formula for the normalization map $C\arrow X_{\overline{k}}$,
this map sends $dx/y$ to $(x-u)^{(p-1)/2}dx/y$.

To show that $N\subset H^1(X,O)\otimes_k \overline{k}$
has zero intersection with the $k$-linear space $H^1(X,O)$,
it is equivalent to show that the coefficients 
$a_0,\ldots,a_{(p-1)/2}\in \overline{k}$ of $(x-u)^{(p-1)/2}dx/y$
in terms of our $k$-basis for $H^0(X,K_X)$ are $k$-linearly independent.
These coefficients are $\binom{(p-1)/2}{i}(-u)^{(p-1)/2-i}$ for
$0\leq i\leq (p-1)/2$. Since nonzero factors in $k$ do not matter,
it suffices to show that $1,u,u^2,\ldots,u^{(p-1)/2}\in \overline{k}$
are $k$-linearly independent. Since $t\in k$
is not a $p$th power, $u=t^{1/p}$ has degree $p$ over $k$,
and so even $1,u,u^2,\ldots,u^{p-1}$ are $k$-linearly independent.
This completes the proof that $G=\Pic^0_{X/k}$ is a pseudo-abelian
variety over $k$.
\end{proof}

We remark that for any odd prime $p$,
the genus $(p+1)/2$ in this example is the smallest possible
for a geometrically integral projective curve $X$ 
over a field $k$ of characteristic $p$
whose Jacobian $G$ is a pseudo-abelian variety over $k$
but not an abelian variety. Indeed, such a curve $X$ must be regular;
otherwise the kernel $K$ of the homomorphism from $G$ to the Jacobian
of the normalization of $X$ would be a nontrivial smooth connected
affine $k$-subgroup of $G$. (It suffices to check that $K_{\overline{k}}$
is a nontrivial smooth connected affine group over $\overline{k}$. 
To do that, let $f:D\arrow X$ be the normalization; this is not
an isomorphism if $X$ is not regular.
Then $f:D_{\overline{k}}\arrow X_{\overline{k}}$ is a birational
morphism of (possibly singular) integral projective curves. 
The kernel $K_{\overline{k}}$ of the surjection
$G_{\overline{k}}=\Pic^0_{X/\overline{k}}\arrow \Pic^0_{D/\overline{k}}$
\def\Gm{\mathbf{G}_m}
has $K(\overline{k})=H^0(X_{\overline{k}},(R_{D/X}\mathbf{G}_{m,D})/
\mathbf{G}_{m,X})$, which is
nontrivial if $f$ is not an isomorphism. More precisely, $K_{\overline{k}}$
is a quotient of the product of the groups $(O_{D,y}/\m^N)^*$ viewed
as $\overline{k}$-groups for a nonempty finite set of
points $y\in D(\overline{k})$ and some positive integers $N$,
and so $K_{\overline{k}}$ is smooth, connected,
and affine over $\overline{k}$.)

Next, $X$ is not smooth over $k$; otherwise its Jacobian
$G$ would be an abelian variety. So the geometric genus of $X_{\overline{k}}$
(the genus of the normalization of $X_{\overline{k}}$)
is less than the arithmetic genus of $X_{\overline{k}}$
(or equivalently
of $X$), by considering the exact sequence of sheaves
$0\arrow O_{X_{\overline{k}}}\arrow g_* O_{C}\arrow L\arrow 0$
associated to the normalization $g\colon C\arrow X_{\overline{k}}$.
Finally, the geometric genus of
$X_{\overline{k}}$ is not zero (otherwise $G$ would be affine; the Jacobian
$\Pic^0_{X/\overline{k}}$ is an extension of the Jacobian of the normalization
by a smooth connected affine group). Tate showed
that the geometric and arithmetic genera
differ by a multiple of $(p-1)/2$ for a geometrically integral
regular projective curve $X$ over a field of characteristic $p$
\cite{Tategenus}; see Schr\"oer \cite{Schroer}
for a proof in the language of schemes.
So $X$ must have
arithmetic genus at least $1+(p-1)/2=(p+1)/2$, as claimed.

\section{Mordell-Weil theorem for pseudo-abelian varieties}

One can expect many properties of abelian varieties to extend
to pseudo-abelian varieties. We show here that
the Mordell-Weil theorem holds for pseudo-abelian varieties.

\begin{proposition}
\label{mordell}
Let $E$ be a pseudo-abelian variety over a field $k$ which
is finitely generated over the prime field.
Then the abelian group $E(k)$ is finitely generated.
\end{proposition}

\begin{proof}
If $k$ has characteristic zero, then $E$ is an abelian variety
and this is the usual Mordell-Weil theorem \cite[Chapter 6]{Lang}.
So let $k$ be a finitely generated field over $\F_p$.
As with any connected group scheme over $k$, we can write
$E$ as an extension
$$1\arrow H\arrow E\arrow B\arrow 1$$
with $H$ a connected affine $k$-group scheme and $B$ an abelian variety
\cite[Lemme IX.2.7]{Raynaud}.
Since $E$ is a pseudo-abelian variety, $E$
is commutative and the maximal smooth connected $k$-subgroup of $H$ is trivial.
Note that we can define the maximal smooth $k$-subgroup
of any $k$-group scheme $H$ as the Zariski closure
of the group $H(k_s)$ \cite[Lemma C.4.1]{CGP}.
So $H(k_s)$ is finite, and so $H(k)$ is finite.
By the exact sequence $H(k)\arrow E(k)\arrow B(k)$,
where $B(k)$ is finitely generated by Mordell-Weil,
$E(k)$ is finitely generated.
\end{proof}

\section{Birational characterization of pseudo-abelian varieties}

In this section we show that pseudo-abelian varieties
can be characterized among all smooth algebraic groups
without using the group structure. In fact,
the birational equivalence class of a smooth connected group $G$
over a field $k$
is enough to determine whether $G$ is a pseudo-abelian variety.
This makes pseudo-abelian varieties a very natural class
of algebraic groups.
Theorem \ref{bir} says
that a smooth connected $k$-group is pseudo-abelian if and only if 
it is not ``smoothly uniruled'', a notion which we will define. 

As usual,
a variety $X$ over a field $k$ is {\it uniruled }if
there is a variety $Y$ over $k$ and a dominant rational map $Y\times \P^1
\dashrightarrow X$ over $k$ which does not factor
through $Y$ \cite[Proposition IV.1.3]{Kollar}.
We say that a variety $X$
is {\it rationally connected }if a compactification of $X$
is rationally
connected in the usual sense \cite[Definition IV.3.2.2]{Kollar}. Equivalently,
$X$ is rationally connected if and only if there is a variety $Y$ over $k$
and a rational map $u\colon
Y\times \P^1\dashrightarrow X$ over $k$ such that the
associated map $u^{(2)}\colon  Y\times \P^1\times \P^1\dashrightarrow X
\times_k X$ is dominant. Next, a variety $X$ over a field $k$
is {\it generically
smooth }if the smooth locus of $X$ over $k$ is nonempty. Over a perfect
field, every variety is generically smooth. 

We now make a new definition. A generically smooth
variety $X$ over a field $k$ is {\it smoothly uniruled }if there
are generically smooth $k$-varieties $B$ and $E$ with
dominant rational maps
$$ \xymatrix{
E \ar@{-->}[r]\ar@{-->}[d] &X\\
B & }$$
such that the generic
fiber of $E\dashrightarrow B$ is a generically smooth and rationally connected
variety over $k(B)$,
and $E\dashrightarrow X$ does not factor through $B$.
Smooth uniruledness
depends only on the birational equivalence class of $X$ over $k$.

It is clear that a smoothly uniruled variety is
uniruled. The converse holds for $k$ perfect, but not in general,
as Theorem \ref{bir} will show. (Being ``smoothly uniruled''
does not imply being ``separably uniruled'',
which is stronger than uniruledness
even over an algebraically closed field of positive characteristic
\cite[Definition IV.1.1]{Kollar}.)
Note that uniruledness is a geometric notion; that is,
a $k$-variety $X$ (which need not be generically smooth)
is uniruled if and only if $X_{\overline{k}}$
has uniruled irreducible components \cite[Proposition IV.1.3]{Kollar}.
That is not true for smooth uniruledness (over an imperfect field $k$),
as Theorem \ref{bir} will imply.
At least smooth uniruledness
does not change under separable algebraic extensions of $k$.
Since smooth uniruledness turns out to be an interesting property
of algebraic groups, it should be worthwhile to study
smooth uniruledness for other classes of varieties over imperfect
fields.

\begin{theorem}
\label{bir}
Let $G$ be a smooth connected group over a field $k$.
Then $G$ is an abelian variety if and only if $G$
is not uniruled. And $G$ is a pseudo-abelian variety if and only if $G$
is not smoothly uniruled. In particular, whether $G$ is a pseudo-abelian
variety depends only on the birational equivalence class of $G$
over $k$.
\end{theorem}

\begin{proof}
If $G$ is not an abelian variety, then $G_{\overline{k}}$
has a nontrivial smooth connected affine normal subgroup $N$ 
over $\overline{k}$, by Chevalley's theorem. Such a group $N$ is rational
\cite[Remark 14.14]{Borel} and has positive dimension.
Using the product map $G_{\overline{k}}
\times N \arrow G_{\overline{k}}$,
it follows that $G_{\overline{k}}$ is uniruled.
Equivalently, $G$ is uniruled.
Conversely, if $G$ is an abelian variety, then $G_{\overline{k}}$
contains no rational curves, and so $G$ is not uniruled.

If $G$ is not a pseudo-abelian variety, then $G$ has a nontrivial
smooth connected affine normal $k$-subgroup $N$. Then $N_{\overline{k}}$
is rational and so $N$ is rationally connected, as that is a geometric
property \cite[Ex.~IV.3.2.5]{Kollar}. The diagram
$$\xymatrix{
G\times N \ar[r]^{gn} \ar[d]_{g} & G\\
G & }$$
has the properties needed to show that $G$ is smoothly uniruled:
the base variety $G$ is generically smooth, the generic fiber $N_{k(G)}$
of the vertical map is generically smooth and rationally connected,
and the horizontal map $G\times N\arrow G$
is dominant and does not factor through the vertical map.

Conversely, let $G$ be a pseudo-abelian variety over a field $k$. Suppose that
$G$ is smoothly uniruled. Let
$$\xymatrix{
E \ar@{-->}[r]\ar@{-->}[d] &G\\
B & }$$
be a diagram as in the definition of smooth uniruledness.
Thus the generic
fiber of $E\dashrightarrow B$ is a generically smooth and rationally connected
variety over $k(B)$,
and $E\dashrightarrow G$ does not factor through $B$. It follows that
these properties hold over a dense open subset of $B$.
Because $B$ is a generically smooth $k$-variety, $B(k_s)$ is Zariski dense
in $B$. So there is a point in $B(k_s)$ whose inverse image $Y$
in $E$ is a generically smooth,
rationally connected variety over $k_s$
with a nonconstant rational map $f\colon Y\dashrightarrow
G_{k_s}$. In particular, $Y$ has positive dimension.
Here $Y(k_s)$ is Zariski dense in $Y$
because $Y$ is generically smooth.

By Theorem \ref{pseudo}, we can write
the pseudo-abelian variety $G$ as $(A\times H)/K$
for some abelian variety $A$, commutative affine $k$-group scheme $H$,
and commutative finite $k$-group scheme $K$. The image of the rationally
connected $k_s$-variety $Y$ in the abelian variety $A/K$ must be
a $k_s$-rational point. So $f$ maps $Y$ into the inverse image
of this point in $G_{k_s}$, which is a principal $H_{k_s}$-bundle
over $\Spec(k_s)$. Since $Y(k_s)$ is Zariski dense in $Y$,
this principal bundle
has a $k_s$-rational point and hence is trivial. Thus we get
a nonconstant rational map from the generically smooth
variety $Y$ to $H_{k_s}$. It follows that $H(k_s)$ is infinite,
and so the maximal smooth $k$-subgroup of $H$ has positive dimension.
Such a subgroup is affine and contained in $G$, 
contradicting that $G$ is a pseudo-abelian variety.
\end{proof}

\section{Construction of pseudo-abelian varieties: supersingular case}

The unipotent quotient of a pseudo-abelian variety
over a field $k$
is a smooth connected commutative unipotent group over $k$.
In this section, we show that when $k$ is imperfect of
characteristic $p$,
every smooth connected commutative
group of exponent $p$ over $k$ occurs as the unipotent quotient
of some pseudo-abelian variety $E$, even in the special case
where the abelian subvariety of $E$ is a supersingular elliptic curve
(Corollary \ref{supersingular}). Thus there are many more pseudo-abelian
varieties over an imperfect field than Raynaud's original
examples, Weil restrictions of abelian varieties. (Weil restrictions
occur only in certain dimensions. For example,
if a Weil restriction $R_{l/k}B$ for a purely inseparable extension
$l/k$ has its maximal abelian subvariety
of dimension 1, then the abelian variety $B$
has dimension 1, and so the unipotent quotient
of $R_{l/k}B$ has dimension $p^r-1$ for some $r$.)

\begin{definition}
\label{highdef}
Let $U$ be a smooth connected commutative unipotent group over a field $k$.
Let $K$ be a finite commutative $k$-group scheme.
We say that a commutative extension
$$1\arrow K\arrow H\arrow U \arrow 1$$
is {\it highly nontrivial }if the maximal smooth connected $k$-subgroup of $H$
(which is necessarily unipotent)
is trivial.
\end{definition}

For us, the point of the notion of highly nontrivial extensions is:

\begin{lemma}
\label{highly}
(1) Let $1\arrow K\arrow H\arrow U\arrow 1$ be a highly nontrivial
extension of a smooth connected commutative unipotent
group $U$ over a field $k$.
Let $A$ be an abelian variety over $k$ that contains $K$ as a subgroup scheme.
Then $E:=(A\times H)/K$ is a pseudo-abelian variety which is an extension
$$1\arrow A\arrow E\arrow U\arrow 1.$$

(2) Conversely, let $E$ be any pseudo-abelian variety over a field $k$
of characteristic $p$.
Write $E$ as an extension
$1\arrow A\arrow E\arrow U\arrow 1$ with $A$ an abelian
variety and $U$ a smooth connected commutative unipotent group.
Let $p^r$ be the exponent of $U$. Then $E$ can be written
as $(A\times H)/A[p^r]$ for some
highly nontrivial extension $1\arrow A[p^r]
\arrow H\arrow U\arrow 1$.
\end{lemma}

\begin{proof}
Let us prove (1). Clearly $E$ is
an extension $1\arrow A \arrow E\arrow U\arrow 1$.
It follows that $E$ is a smooth connected $k$-group. Clearly $E$
is commutative.

Let $N$ be a smooth connected affine $k$-subgroup of $E$.
Then $N$ must map trivially into the abelian variety $E/H=A/K$.
Therefore $N$ is contained in the subgroup scheme $H$ of $E$. Since
$H$ is a highly nontrivial extension, $N$ is trivial. Thus $E$ is
a pseudo-abelian variety, proving (1).

We turn to (2). Since $U$ has exponent $p^r$, the abelian group $\Ext^1(U,A)$
is killed by $p^r$. Consider the exact sequence
$$\Ext^1(U,A[p^r])\arrow \Ext^1(U,A)\stackrel{p^r}{\longrightarrow}
\Ext^1(U,A).$$
(Such exact sequences hold
for $\Ext$ in any abelian category, in this case the category
of commutative $k$-group schemes of finite type
\cite[Th\'eor\`eme VIA.5.4.2]{SGA3}.)
The exact sequence shows 
that the extension $E$ comes from a commutative extension
$1\arrow A[p^r]\arrow H\arrow U\arrow 1$, with $H\subset E$.
Clearly $H$ is affine. Since $E$ is a pseudo-abelian variety,
the maximal smooth connected $k$-subgroup of $H$ is trivial,
and so $H$ is a highly nontrivial extension.
\end{proof}

\begin{lemma}
\label{alphap}
Let $U$ be a smooth connected commutative group
of exponent $p$ over a field $k$ of characteristic $p$. If $k$ is imperfect,
then there is a highly nontrivial extension of $U$
by $\alpha_p$.
\end{lemma}

\begin{proof}
It suffices to show that there are highly nontrivial extensions
of $(\Ga)^s$ by $\alpha_p$ over $k$ for some arbitrarily large numbers $s$.
Indeed, having a highly nontrivial extension is a property which
passes from one smooth connected commutative unipotent $k$-group
to any smooth connected $k$-subgroup. And every smooth connected
commutative $k$-group of exponent $p$ and dimension $n$
is isomorphic to the subgroup of $(\Ga)^{n+1}$ defined 
by some $p$-polynomial over $k$.

Since $k$ is imperfect, we can choose an element
$t$ in $k^*$ which is not a $p$th power.
We will exhibit a highly nontrivial extension
$1\arrow \alpha_p\arrow H\arrow (\Ga)^{(p-1)p^{r-1}}\arrow 1$
over $k$,
for any $r\geq 1$. 
For clarity, first take $r=1$. That is, we want
to construct a highly nontrivial extension
$1\arrow \alpha_p\arrow H\arrow (\Ga)^{p-1}\arrow 1$.
Let $l=k(u)$ where $u=t^{1/p}$; thus $l$ is a field of degree $p$
over $k$. We will take $H$ to be the Weil restriction
$R_{l/k}\alpha_p$. A general reference on Weil restriction
is \cite[Appendix A.5]{CGP}. Note that Weil restriction need not multiply
dimensions by $p=[l\colon k]$ for non-smooth schemes such as the 0-dimensional
scheme $\alpha_p$.
In fact, $R_{l/k}\alpha_p$ has dimension $p-1$; explicitly, it is the
$k$-subgroup scheme
$$\{ (a_0,a_1,\ldots,a_{p-1})\in (\Ga)^{p}: a_0^p+t
a_1^p+\cdots + t^{p-1}a_{p-1}^p=0\},$$
as we find by writing out the equation $(a_0+a_1u+\ldots
+a_{p-1}u^{p-1})^p=0$. We check immediately that the kernel of
the natural homomorphism
$R_{l/k}\alpha_p\arrow (R_{l/k}\Ga)/\Ga$
is $\alpha_p$. The resulting injection
$$(R_{l/k}\alpha_p)/\alpha_p\arrow (R_{l/k}\Ga)/\Ga$$
is an isomorphism, because the two $k$-group schemes
have the same dimension and $(R_{l/k}\Ga)/\Ga$ is smooth and connected.
The quotient group $(R_{l/k}\Ga)/\Ga$
is isomorphic to $(\Ga)^p/\Ga\cong (\Ga)^{p-1}$. Thus $H$ is
an extension of $(\Ga)^{p-1}$ by $\alpha_p$, as we want. By
construction, $H$ is commutative of exponent $p$.

It remains to show that the maximal smooth connected $k$-subgroup
of $H$ is trivial. We have $H(k_s)=(R_{l/k}\alpha_p)(k_s)=\alpha_p(l_s)=0$,
since $l_s$ is a field of characteristic $p$. Therefore
every smooth $k$-subgroup of $H$, connected or not, is trivial.
So $H$ is a highly nontrivial extension as we want, in the case $r=1$.

We now generalize the construction to exhibit
a highly nontrivial extension
$1\arrow \alpha_p\arrow H\arrow (\Ga)^{(p-1)p^{r-1}}\arrow 1$
over $k$
for any $r\geq 1$. Again, let $t$ be an element of $k^*$
which is not a $p$th power.

Let $u=t^{1/p}$, $v=t^{1/p^{r-1}}$, and $w=t^{1/p^r}$.
First define 
\begin{align*}
U& :=(R_{k(w)/k}\Ga)/(R_{k(v)/k}\Ga)\\
&=R_{k(v)/k}((R_{k(w)/k(v)}\Ga)/\Ga)\\
&\cong (\Ga)^{(p-1)p^{r-1}}.
\end{align*}
We define an extension group $1\arrow
\alpha_p\arrow H\arrow U\arrow 1$ as the fiber product
\begin{align*}
H& := [R_{k(v)/k}((R_{k(w)/k(v)}\Ga)/\Ga)]\times_{(R_{k(u)/k}\Ga)/\Ga}
R_{k(u)/k}\alpha_p\\
& = U\times_{(\Ga)^{p-1}}R_{k(u)/k}\alpha_p.
\end{align*}
Here the homomorphism
$R_{k(v)/k}((R_{k(w)/k(v)}\Ga)/\Ga)\arrow (R_{k(u)/k}\Ga)/\Ga$
on the left corresponds on $k$-rational
points to taking the $p^{r-1}$st power, and the homomorphism
on the right is $R_{k(u)/k}\alpha_p\arrow (R_{k(u)/k}\alpha_p)/\alpha_p
=(R_{k(u)/k}\Ga)/\Ga$. Since the latter homomorphism is a surjection
with kernel $\alpha_p$, it is clear that $H$ is an extension
$1\arrow \alpha_p\arrow H\arrow U\arrow 1$. The definition shows
that $H$
is commutative of exponent $p$.

It remains to show that $H$ is a highly nontrivial extension.
We will prove the stronger
statement that the maximal smooth $k$-subgroup of $H$ is trivial.
That holds if $H(k_s)=1$. By definition of $H$,
$H(k_s)$ is the fiber product
$$H(k_s)=[k_s(w)/k_s(v)]\times_{k_s(u)/k_s} \alpha_p(k_s(u)),$$
where the left homomorphism is the $p^{r-1}$st power.
Since $k_s(u)$ is a field of characteristic $p$, $\alpha_p(k_s(u))=0$.
So $H(k_s)=\{y\in k_s(w)/k_s(v): y^{p^{r-1}}\in k_s\}$.
This group is zero by the following lemma, applied to the field $F=k_s$.

\begin{lemma}
\label{intersection}
Let $F$ be a field of characteristic $p>0$ with an element
$t\in F$ that is not a $p$th power in $F$. Let $r$ be a positive integer.
Let $u=t^{1/p}$, $v=t^{1/p^{r-1}}$, and $w=t^{1/p^r}$. Then
$$F(w)\cap F^{1/p^{r-1}}=F(v).$$
\end{lemma}

\begin{proof}
The intersection $F(w)\cap F^{1/p^{r-1}}$ is a subfield of
$F(w)$ that contains $F(v)$. It is equal to $F(v)$ because
$[F(w)\colon F(v)]=p$ is prime and $w$ is not in $F^{1/p^{r-1}}$
(because $w^{p^r}=t$ and $t$ is not a $p$th power in $F$).
\end{proof}

Thus $H(k_s)=0$. We have shown that
the extension $1\arrow K\arrow H\arrow (\Ga)^{p^{r-1}(p-1)}
\arrow 1$ is highly nontrivial, proving Lemma \ref{alphap}.
\end{proof}

\begin{corollary}
\label{supersingular}
For any smooth connected commutative group $U$
of exponent $p$
over an imperfect field $k$ and any supersingular elliptic
curve $A$ over $k$, there is a pseudo-abelian variety over $k$
which is an extension of $U$ by $A$.
\end{corollary}

The assumption that $k$ is imperfect is essential, by Chevalley's theorem:
every pseudo-abelian variety over a perfect field is an abelian variety.
In particular, there is a pseudo-abelian variety
$1\arrow A\arrow E\arrow \Ga \arrow 1$ over $k$ with $A$ a supersingular
elliptic curve whenever $k$ is imperfect, but not when $k$ is perfect.

\begin{proof}
Let $A$ be a supersingular elliptic curve over $k$. Then
the kernel of the Frobenius homomorphism on $A$ is isomorphic
to $\alpha_p$ over $k$. Since $k$ is imperfect, Lemma \ref{alphap}
shows that there is a highly nontrivial extension $H$ of $U$ by
$\alpha_p$. By Lemma \ref{highly}(1), $E=(A\times H)/\alpha_p$
is a pseudo-abelian variety. It is an extension
of $U$ by $A$.
\end{proof}

\section{Construction of pseudo-abelian varieties: ordinary case}
\label{ordsection}

This section shows again that there are many pseudo-abelian
varieties over an imperfect field $k$. Namely, for any ordinary
elliptic curve $A$ over $k$ which cannot be defined over
the subfield $k^p$,
every smooth connected commutative
group of exponent $p$ over $k$ occurs as the unipotent quotient
of a pseudo-abelian variety with abelian subvariety $A$,
possibly after a finite separable extension of $k$
(Corollary \ref{ordinary}). This is somewhat harder than
the analogous result for supersingular elliptic
curves, Corollary \ref{supersingular}. The analysis leads to
a conjectural computation of $\Ext^2_k(\Ga,\Gm)$ by generators
and relations (Question \ref{ext2}).

The situation is different for pseudo-abelian varieties $E$ over $k$
whose abelian subvariety is an ordinary elliptic curve
which can be defined over $k^p$. In that case, the unipotent
quotient of $E$ is very restricted, by Lemma \ref{equiv}
and Example \ref{dim2}.

\begin{lemma}
\label{nontrivial}
Let $k$ be a field of characteristic $p>0$. Let $K$ be
a commutative $k$-group scheme which is a nontrivial extension of $\Z/p$
by $\mu_p$. Let $U$ be a smooth connected commutative
$k$-group of exponent $p$. Then there is a highly
nontrivial extension of $U$ by $K$ over $k$.
\end{lemma}

\begin{proof}
As in the proof of Lemma \ref{alphap},
it suffices to show that there are highly nontrivial extensions
of $(\Ga)^s$ by $K$ over $k$ for some arbitrarily large numbers $s$.

We will exhibit a highly nontrivial extension
$1\arrow K\arrow H\arrow (\Ga)^{(p-1)p^{r-1}}\arrow 1$
over $k$,
for any $r\geq 1$. We are assuming that the class of $K$
in $\Ext^1(\Z/p,\mu_p)=\Ext^1(\Z/p,\Gm)=k^*/(k^*)^p$
\cite[Corollaire III.6.4.4]{DG}
is nontrivial.
(Here $\Ext$ is taken in the abelian category
of commutative $k$-group schemes of finite type.)
Let $t\in k^*$ represent this extension; then
$t$ is not a $p$th power in $k$.

For clarity, first take $r=1$. That is, we want
to construct a highly nontrivial extension
$1\arrow K\arrow H\arrow (\Ga)^{p-1}\arrow 1$.
Let $l=k(u)$ where $u=t^{1/p}$; thus $l$ is a field of degree $p$
over $k$. We will take $H$ to be the Weil restriction
$R_{l/k}\mu_p$. Like the Weil restriction
$R_{l/k}\alpha_p$ in the proof of Lemma \ref{alphap},
$R_{l/k}\mu_p$ has dimension $p-1$; explicitly, it is the
hypersurface 
$$\{ (a_0,a_1,\ldots,a_{p-1})\in \A^{p}_k: a_0^p+t
a_1^p+\cdots + t^{p-1}a_{p-1}^p=1\},$$
as we find by writing out the equation $(a_0+a_1u+\ldots
+a_{p-1}u^{p-1})^p=1$. It is straightforward to check that the natural
homomorphism
$$(R_{l/k}\mu_p)/\mu_p\arrow (R_{l/k}\Gm)/\Gm$$
is an isomorphism. The quotient group $(R_{l/k}\Gm)/\Gm$
is a smooth connected commutative group of exponent $p$,
described explicitly as the subgroup
$$U:= \{(x_0,\ldots,x_{p-1})\in (\Ga)^p: x_0^p+tx_1^p+\cdots
+t^{p-1}x_{p-1}^p=x_{p-1}\}$$
\cite[Proposition VI.5.3]{Oesterle}.
Under Oesterl\'e's isomorphism $(R_{l/k}\Gm)/\Gm\arrow U$,
the point $u=t^{1/p}$ in $(R_{l/k}\Gm)(k)=l^*$ maps to
$(0,\ldots,0,1/t)$ in $U(k)$.

The homomorphism $f\colon U\arrow (\Ga)^{p-1}$ given by
$(x_0,\ldots,x_{p-1})\mapsto (x_0,\ldots x_{p-2})$
has kernel isomorphic to $\Z/p$, generated by the point
$(0,\ldots,0,1/t)$. By counting dimensions, it follows
that $f$ is surjective and gives an isomorphism
$U/(\Z/p)\cong (\Ga)^{p-1}$. Therefore $H=R_{l/k}\mu_p$
is a three-step extension $\begin{pmatrix} (\Ga)^{p-1}\\
\Z/p\\ \mu_p \end{pmatrix}$. (The notation means that $H$ maps onto
the top group $(\Ga)^{p-1}$, the kernel maps onto the middle group
$\Z/p$, and so on.) Write $K_1$ for the subgroup
$\begin{pmatrix} \Z/p\\
\mu_p \end{pmatrix}$ in $H$. 

We want to show that $K_1$ is the nontrivial
extension classified by $t\in \Ext^1(\Z/p,\mu_p)=(k^*)/(k^*)^p$.
We can use that $\Ext^1(\Z/p,\mu_p)$ maps isomorphically
to $\Ext^1(\Z/p,\Gm)$, by the exact sequence
$$\Hom(\Z/p,\Gm)\arrow \Ext^1(\Z/p,\mu_p)\arrow \Ext^1(\Z/p,\Gm)
\stackrel{p}{\longrightarrow} \Ext^1(\Z/p,\Gm).$$
(Here $\Hom(\Z/p,\Gm)=\mu_p(k)=1$, and
multiplication by $p$ is zero on $\Ext^1(\Z/p,\Gm)$ because
the group $\Z/p$ is killed by $p$.)
The corresponding extension $W$
of $\Z/p$ by $\Gm$ is the inverse image of $\Z/p\subset U$
under the surjection $R_{l/k}\Gm\arrow U$. The extension
$1\arrow \Gm\arrow W\arrow \Z/p\arrow 1$
is classified by the element of $k^*/(k^*)^p$ which is the
$p$th power of any element of $W(k)$ that maps to $1\in \Z/p$.
As we have said, the element $u\in (R_{l/k}\Gm)(k)=l^*$ maps to
$1\in \Z/p$, and its $p$th power is $t$. So $K_1$ is the nontrivial
extension $K$ classified by $t\in k^*/(k^*)^p$, as we want.

It remains to show that $H$ is a highly nontrivial
extension of $(\Ga)^{p-1}$ by $K$. We have $H(k_s)=\mu_p(k_s(t^{1/p}))=1$,
because $k_s(t^{1/p})$ is a field of characteristic $p$.
So the maximal smooth connected $k$-subgroup of $H$ is trivial,
as we want.

We now generalize the construction. Given a nontrivial
extension $K$ of $\Z/p$ by $\mu_p$, we will
exhibit a highly nontrivial extension
$1\arrow K\arrow H\arrow (\Ga)^{(p-1)p^{r-1}}\arrow 1$
over $k$
for any $r\geq 1$. Again, let $t\in k^*$ represent
the class of $K$ in $\Ext^1(\Z/p,\mu_p)\cong (k^*)/(k^*)^p$.

Let $u=t^{1/p}$, $v=t^{1/p^{r-1}}$, and $w=t^{1/p^r}$.
Our extension $U_r=
\begin{pmatrix} (\Ga)^{(p-1)p^{r-1}}\\
\Z/p \end{pmatrix}$ will be
\begin{align*}
U_r& :=(R_{k(w)/k}\Gm)/(R_{k(v)/k}\Gm)\\
&=R_{k(v)/k}((R_{k(w)/k(v)}\Gm)/\Gm).
\end{align*}
The second description shows
that $U_r$ is a smooth connected commutative $k$-group
of exponent $p$ and dimension $(p-1)p^{r-1}$.
(Indeed, $(R_{k(w)/k(v)}\Gm)/\Gm$ is essentially
the $(p-1)$-dimensional unipotent group considered above, but
over $k(v)$ instead of $k$.) This description gives
equations for $U_r$:
$$U_r\cong \{(x_0,\ldots,x_{p-1})\in (R_{k(v)/k}\Ga)^p:
x_0^p+vx_1^p+\cdots+v^{p-1}x_{p-1}^p=x_{p-1} \}.$$

Define a homomorphism $f\colon U_r\arrow (\Ga)^{(p-1)p^{r-1}}$ over $k$ by
$(x_0,\ldots,x_{p-1})\mapsto (x_0,\ldots,x_{p-2})$. (Here
each $x_i$ is in $R_{k(v)/k}\Ga\cong (\Ga)^{p^{r-1}}$.) The kernel
of $f$ is the $k$-subgroup $\Z/p$ of $U_r$
generated by $(x_0,\ldots,x_{p-1})=
(0,\ldots,0,1/v)$. We noted in the case $r=1$ that
the isomorphism $(R_{l/k}\Gm)/\Gm\arrow U$
sends the point $u=t^{1/p}$ in $(R_{l/k}\Gm)(k)=l^*$ to
$(0,\ldots,0,1/t)$ in $U(k)$. As a result, 
the point $(0,\ldots,0,1/v)$ in $U_r(k)$ is the image
of the point $w$ in $(R_{k(w)/k}\Gm)(k)=k(w)^*$ under the identification
$U_r=(R_{k(w)/k}\Gm)/(R_{k(v)/k}\Gm)$.
By counting dimensions, $f$ is surjective, and so
$U$ is an extension $\begin{pmatrix} (\Ga)^{(p-1)p^{r-1}}\\
\Z/p \end{pmatrix}$.

The extension of $U_r$ by $\Gm$ we consider is the fiber product
\begin{align*}
E & :=[(R_{k(w)/k}\Gm)/(R_{k(v)/k}\Gm)]\times_{(R_{k(u)/k}\Gm)/\Gm}
R_{k(u)/k}\Gm\\
& = U_r \times_{(R_{k(u)/k}\Gm)/\Gm}
R_{k(u)/k}\Gm,
\end{align*}
where the homomorphism $(R_{k(w)/k}\Gm)/(R_{k(v)/k}\Gm)\arrow
(R_{k(u)/k}\Gm)/\Gm$ corresponds on $k$-rational
points to taking the $p^{r-1}$st power. This extension
comes from an extension $H$ of $U_r$ by $\mu_p$,
$$H:=U_r\times_{(R_{k(u)/k}\Gm)/\Gm}
R_{k(u)/k}\mu_p,$$
since $(R_{k(u)/k}\mu_p)/\mu_p$ is isomorphic to
$(R_{k(u)/k}\Gm)/\Gm$. Thus $H$ is a three-step extension
$\begin{pmatrix} (\Ga)^{(p-1)p^{r-1}}\\
\Z/p\\ \mu_p \end{pmatrix}$.

Let $K_1$ be the subgroup $\begin{pmatrix} \Z/p\\ \mu_p \end{pmatrix}$
in $H$. We want to show that $K_1$ is the nontrivial extension
of $\Z/p$ by $\mu_p$ corresponding to $t\in k^*/(k^*)^p=
\Ext^1(\Z/p,\mu_p)$. It is equivalent to show that the inverse
image $L_1$ of $\Z/p\subset U_r$ in $E$ is the extension of $\Z/p$ by $\Gm$
corresponding to $t\in k^*/(k^*)^p=\Ext^1(\Z/p,\Gm)$. As we have
computed, $L_1$ contains the $k$-rational point $w$
in $(R_{k(w)/k}\Gm)/(R_{k(v)/k}\Gm)$. The image of $w$
under the ``$p^{r-1}$st power homomorphism'' to $(R_{k(u)/k}\Gm)/\Gm$
is clearly the image of $u$ in $(R_{k(u)/k}\Gm)(k)=k(u)^*$.
The $p$th power of $u$ in $L_1$ is the point $t\in \Gm(k)=k^*$,
which shows that the class of the extension $L_1$ is
$t\in k^*/(k^*)^p$. So $K_1$ is isomorphic to the extension $K$
of $\Z/p$ by $\mu_p$ classified by $t$, as we want.

It remains to show that
the extension $1\arrow K\arrow H\arrow (\Ga)^{p^{r-1}
(p-1)}\arrow 1$ is highly nontrivial. We will prove the stronger
statement that $H(k_s)=1$. By definition of $H$,
$H(k_s)$ is the fiber product
$$H(k_s)=[k_s(w)^*/k_s(v)^*]\times_{k_s(u)^*/(k_s)^*} \mu_p(k_s(u)),$$
where the left homomorphism is the $p^{r-1}$st power.
Since $k_s(u)$ is a field of characteristic $p$, $\mu_p(k_s(u))=1$.
So $H(k_s)=\{y\in k_s(w)^*/k_s(v)^*: y^{p^{r-1}}\in (k_s)^*\}$.
We have $H(k_s)=1$ because $k_s(w)\cap (k_s)^{1/p^{r-1}}
=k_s(v)$ (Lemma \ref{intersection}). Thus
the extension $1\arrow K\arrow H\arrow (\Ga)^{p^{r-1}(p-1)}
\arrow 1$ is highly nontrivial.
\end{proof}

The following lemma is a variant of \cite[Proposition 12.2.7]{KM}.

\begin{lemma}
\label{splitting}
Let $A$ be an ordinary elliptic curve over a field $k$ of characteristic
$p>0$. Then the $p$-torsion subgroup scheme $A[p]$ is an extension
of a $k$-form of $\Z/p$ by a $k$-form of $\mu_p$.
The elliptic curve $A$ can be defined over the subfield $k^p$ if and only if
this extension is split.
\end{lemma}

\begin{proof}
The first statement is clear from the fact that $A_{\overline{k}}[p]$
is isomorphic to $\mu_p\times \Z/p$.

Write $G^{(p)}$ for the group scheme
over $k^p$ which is associated to a $k$-group scheme $G$ via the isomorphism
$k\stackrel{\cong}{\longrightarrow} k^p$,
$x\mapsto x^p$. Then the relative Frobenius for $A$ is a homomorphism
$F\colon A\arrow (A^{(p)})_k.$
Define the Verschiebung $V\colon (A^{(p)})_k\arrow A$ to be the dual isogeny.
Since $VF=p$, where $\ker(F)\subset A$ is a $k$-form of $\mu_p$,
$\ker(V)\subset (A^{(p)})_k$ must be a $k$-form of $\Z/p$.

If an ordinary elliptic curve $A$ over $k$ can be defined over $k^p$,
then it can be written as $(B^{(p)})_k$ for some elliptic curve $B$ over $k$.
Then $\ker(V)\subset (B^{(p)})_k=A$ is a $k$-form of $\Z/p$. That subgroup
gives
a splitting of the extension $1\arrow \ker(F)\arrow A[p]\arrow A[p]/\ker(F)
\arrow 1$, as we want.

Conversely, let $A$ be an ordinary elliptic curve over $k$
such that that the extension $1\arrow \ker(F)\arrow A[p]\arrow C\arrow 1$
is split, where $\ker(F)$ is a $k$-form of $\mu_p$ and $C$
is a $k$-form of $\Z/p$. A splitting gives an etale $k$-subgroup $C\subset A$
of order $p$. This gives an etale $k$-subgroup $(C^{(p)})_k\subset (A^{(p)})_k$
of order $p$. But the kernel of the Verschiebung $V\colon (A^{(p)})_k\arrow A$
is also an etale $k$-subgroup of order $p$. By our knowledge of
the $p$-torsion of an ordinary elliptic curve,
it follows that $\ker(V)=(C^{(p)})_k\subset (A^{(p)})_k$.
Therefore $V$ gives an isomorphism 
$(A^{(p)}/C^{(p)})_k\stackrel{\cong}{\longrightarrow} A$. So
$A$ comes from the elliptic curve $A^{(p)}/C^{(p)}$ over $k^p$.
\end{proof}

\begin{corollary}
\label{ordinary}
Let $A$ be an ordinary elliptic curve over a field $k$
of characteristic $p$ which cannot be defined over the subfield $k^p$.
Suppose that the connected component of the identity in the $p$-torsion
subgroup scheme $A[p]$ is isomorphic to $\mu_p$ over $k$. (That always
holds after replacing $k$ by some extension field of degree dividing $p-1$.)
Then, for every smooth connected commutative group $U$
of exponent $p$ over $k$, there is a pseudo-abelian variety
which is an extension of $U$ by $A$.
\end{corollary}

\begin{proof}
Since $A$ cannot be defined over $k^p$, the $p$-torsion subgroup
scheme of $A$ is a nontrivial extension $1\arrow \ker(F)
\arrow A[p]\arrow C\arrow 1$ over $k$, by Lemma \ref{splitting}.
Since we assume that $\ker(F)$ is isomorphic to $\mu_p$ over $k$,
the quotient group $C$ is isomorphic to $\Z/p$ over $k$, by the Weil pairing
\cite[section 2.8.2]{KM}.
By Lemma \ref{nontrivial},
there is a highly nontrivial extension $H$ of $U$ by $A[p]$ over $k$.
By Lemma \ref{highly}(1), $E=(A\times H)/A[p]$
is a pseudo-abelian variety over $k$. It is an extension
of $U$ by $A$.
\end{proof}

The proof of Lemma \ref{nontrivial} suggests the following
question about Ext groups in the abelian category of fppf sheaves
over a field $k$ of characteristic $p$, as studied by Breen
\cite{BreenBull, BreenENS}.
There are natural isomorphisms
$\Ext^1_k(\Ga,\Z/p)\cong k^{\perf}=k^{1/p^{\infty}}$
\cite[Proposition III.6.5.4]{DG}
and $\Ext^1_k(\Z/p,\Gm)\cong (k^*)/(k^*)^p$
\cite[Corollaire III.6.4.4]{DG}.
So we have a product map
$$[\cdot,\cdot )\colon k^{\perf}\otimes_{\Z} k^* \arrow \Ext^2_k(\Ga,\Gm).$$
The product
of an element of $\Ext^1_k(\Ga,\Z/p)$
with an element
of $\Ext^1_k(\Z/p,\Gm)$ is zero in $\Ext^2_k(\Ga,\Gm)$
if and only if there is
a three-step extension $\begin{pmatrix} \Ga\\
\Z/p\\ \Gm \end{pmatrix}$ such that the extensions
$\begin{pmatrix}\Ga\\ \Z/p\end{pmatrix}$
and $\begin{pmatrix}\Z/p\\ \Gm \end{pmatrix}$ are the given ones.
This follows from the neat
description of three-step extensions in any abelian category
by Grothendieck \cite[Proposition IX.9.3.8]{SGA7}. The three-step
extensions constructed in the proof of Lemma \ref{nontrivial}
imply the relation $[t^{1/p^r},t)=0$ in $\Ext^2_k(\Ga,\Gm)$
for all $t$ in $k^*$ and all $r\geq 1$. We can also check
that $[s+t,s+t)=[s,s)+[t,t)$ in $\Ext^2_k(\Ga,\Gm)$
for all $s,t$ in $k$ with $s,t,s+t\neq 0$, for example
using the relation to Brauer groups discussed below. So we have
a homomorphism
\begin{multline*}
\varphi\colon  k^{\perf}\otimes_{\Z}k^*/\big( [t^{1/p^r},t)=0 \text{ for
all }t\in k^*\text{ and all }r\geq 1, [s+t,s+t)=[s,s)+[t,t)\big) \\
\arrow \Ext^2_k(\Ga,\Gm).
\end{multline*}

\begin{question}
\label{ext2}
Is $\varphi$ an isomorphism, for every field $k$
of characteristic $p$?
\end{question}

\begin{remark}
If Question \ref{ext2} has a positive answer (about $\Ext^2_k
(\Ga,\Gm)$ in the abelian category
of fppf sheaves), then the same formula holds for Yoneda Ext in the
abelian category
of commutative affine $k$-group schemes of finite type.
The point is that we have natural
maps $\Ext^i_{k-\text{group}}(G,H)\arrow \Ext^i_{k}(G,H)$ for commutative
affine $k$-group schemes $G$ and $H$. These maps are isomorphisms for
$i\leq 1$ \cite[Proposition III.4.1.9]{DG}
and therefore injective for $i=2$. (They are not always surjective
for $i=2$, by Breen \cite{BreenBull}.) The product map above lands
in $\Ext^2_{k-\text{group}}(\Ga,\Gm)$. So if the map $\varphi$ to
$\Ext^2_k(\Ga,\Gm)$ is an isomorphism, then the product
map to $\Ext^2_{k-\text{group}}(\Ga,\Gm)$ is also an isomorphism.
\end{remark}

Question \ref{ext2} would be a very natural calculation. By the discussion
of three-step extensions, the group $\Ext^2_k(\Ga,\Gm)$ comes
up in trying to classify the commutative group schemes
over $k$. (One also encounters the group $\Ext^2_k(\Ga,\mu_p)$,
which is isomorphic to $\Ext^2_k(\Ga,\Gm)$, since
$\Ext^1_k(\Ga,\Gm)=0$ \cite[Th\'eor\`eme XVII.6.1.1]{SGA3}.)
Question \ref{ext2} somewhat resembles the Milnor
conjecture, or more specifically Kato's description of the $p$-torsion
in the Brauer group of a field $k$ of characteristic $p$:
$$\Br(k)[p]\cong k\otimes_{\Z}k^*/([t,t)=0 \text{ for all }t\in k^*,
[s^p,t)=[s,t)\text{ for all }s\in k, t\in k^*)$$
\cite[Lemma 16, p.~674]{Kato}. (There is a similar presentation
of $\Br(k)[p]$ by Witt \cite{Witt}.)
The analogy is explained by Breen's spectral sequence
\cite{BreenBull} (see the proof of Lemma \ref{ext} below),
which gives an isomorphism
$$\Ext^2_k(\Ga,\Gm)\cong \ker(\alpha\colon
\Br(\A^1_k)[p]\arrow \Br(\A^2_k)[p]).$$
Here $\alpha=m^*-\pi_1^*-\pi_2^*$, where $m,\pi_1,\pi_2$ are the morphisms
$\A^2_k\arrow \A^1_k$ which send $(x,y)$ to $x+y,x,y$, respectively.
(This isomorphism sends a symbol $[a^{1/p^r},b)$ in
$\Ext^2_k(\Ga,\Gm)$ to $[ax^{p^r},b)$ in $\Br(\A^1_k)[p]\subset \Br(k(x))[p]$,
for $a\in k$, $r\geq 0$, and $b\in k^*$.)

\section{Pseudo-abelian varieties and commutative pseudo-reductive groups}

In this section, we consider the problem of classifying
pseudo-abelian varieties $E$ over a field $k$ whose
abelian subvariety is an ordinary elliptic curve
which can be defined over $k^p$, the case not considered
in Corollary \ref{ordinary}. This case is very different:
the possible unipotent quotient groups of $E$ are highly restricted.
Lemma \ref{equiv} shows that the possible unipotent quotient
groups in this case are essentially the same as the possible unipotent
quotient groups of commutative pseudo-reductive groups.
Section \ref{pr} gives positive and negative results
about the possible unipotent quotient groups 
of commutative pseudo-reductive groups.

\begin{lemma}
\label{equiv}
Let $A$ be an ordinary elliptic curve over a field $k$
of characteristic $p$. Suppose that $A$ can be defined
over the subfield $k^p$. Suppose that the subgroup
scheme $\ker(F)\subset A$ is isomorphic to $\mu_p$ over $k$,
as can always be arranged after replacing $k$ by a field
extension of degree dividing $p-1$. For a smooth connected
commutative $k$-group $U$ of exponent $p$, the following are
equivalent.

(1) There is a pseudo-abelian variety $E$ which is an extension
$1\arrow A\arrow E\arrow U\arrow 1$ over $k$.

(2) There is a highly nontrivial extension $1\arrow \mu_p
\arrow H\arrow U\arrow 1$ over $k$.

(3) There is a commutative pseudo-reductive group $G$
which is an extension $1\arrow \Gm\arrow G\arrow U\arrow 1$.
\end{lemma}

These three equivalent properties fail for some smooth connected
commutative groups $U$ of exponent $p$. See section
\ref{pr} for positive and negative results. 
Note that there exist ordinary elliptic curves over any field $k$
of characteristic $p$
with $\ker(F)$ isomorphic to $\mu_p$; it suffices to
apply Honda-Tate theory to produce an elliptic curve over $\F_p$
whose Frobenius eigenvalues are the Weil $p$-numbers
$(-1 \pm \sqrt{1-4p})/2$ \cite{TateHonda}.

\begin{proof}
Assume (2). The obvious inclusions $\mu_p\arrow A$ and
$\mu_p\arrow \Gm$ give commutative extensions of $U$ by $A$,
and of $U$ by $\Gm$. The extension of $U$ by $A$ is 
a pseudo-abelian variety by Lemma \ref{highly}(1), giving (1). The
proof of Lemma \ref{highly}(1) also works to show
that the extension $E$ of $U$ by $\Gm$ is pseudo-reductive.
(Given the extension $1\arrow \mu_p\arrow H\arrow U\arrow 1$,
we have $E=(\Gm\times H)/\mu_p$. Any smooth connected unipotent
$k$-subgroup $N$ of $E$ maps trivially into $\Gm/\mu_p\cong \Gm$,
and hence is contained in $H\subset E$. Since $H$ is a highly
nontrivial extension, $N$ is trivial.) That proves (3).

Conversely, if (1) holds, then Lemma \ref{highly}(2) shows that
the extension $1\arrow A\arrow E\arrow U\arrow 1$
comes from a highly nontrivial extension
$1\arrow A[p]\arrow L\arrow U\arrow 1$
over $k$. We are assuming that $\ker(F)\subset A[p]$ is isomorphic
to $\mu_p$ over $k$. By the Weil pairing, it follows that
$A[p]/\ker(F)$ is isomorphic to $\Z/p$ over $k$
\cite[section 2.8.2]{KM}. Since $A$ can be defined
over the subfield $k^p$, Lemma \ref{splitting} shows
that $A[p]$ is isomorphic to $\mu_p\times \Z/p$
over $k$. So $L/(\Z/p)$ is an extension
$1\arrow \mu_p\arrow L/(\Z/p)\arrow U\arrow 1$. Any smooth
connected $k$-subgroup of $L/(\Z/p)$ must be trivial; otherwise
its inverse image in $L$ would be a smooth $k$-group of positive
dimension, contradicting that $L$ is a highly nontrivial extension.
So $L/(\Z/p)$ is a highly nontrivial extension of $U$ by $\mu_p$,
and (2) is proved.
Finally, if (3) holds, then the same proof
as for Lemma \ref{highly}(2) shows that $G$ comes from a highly
nontrivial extension of $U$ by $\mu_p$. That is, (2) holds.
\end{proof}

\section{Commutative pseudo-reductive groups}
\label{pr}

Conrad-Gabber-Prasad have largely reduced the classification
of pseudo-reductive groups over a field $k$ to the case of commutative
pseudo-reductive groups, which seems intractable \cite[Introduction]{CGP}.
This section gives a rough classification
of the commutative pseudo-reductive groups of dimension 2
(Corollary \ref{cordim1}) as well as examples showing
the greater complexity of the problem in higher dimensions.

A commutative pseudo-reductive group over $k$ is an extension of a
smooth connected commutative unipotent group by a torus.
So the main question is which unipotent quotient groups can occur.
This is closely related to the question of which unipotent
quotient groups can occur for
certain pseudo-abelian varieties over $k$, for example those
whose abelian subvariety is an ordinary elliptic curve
which can be defined over $k^p$, by Lemma \ref{equiv}.

For any field $k$, $\Ext^1(\Ga,\Gm)=0$ in the abelian category
of commutative $k$-group schemes. It follows that
the unipotent quotient $U$ of a commutative
pseudo-reductive group must be $k$-wound;
that is, $U$ does not contain the additive group $\Ga$ as a $k$-subgroup.
One main result of this section is that every
$k$-wound group of dimension 1
is the unipotent quotient of some commutative pseudo-reductive
group $E$ over $k$ (Corollary \ref{cordim1}). For $k$ separably
closed, we can take $E$ to have dimension 2.
(For a smooth connected unipotent group of dimension 1 over $k$,
``$k$-wound'' just means ``not isomorphic
to $\Ga$ over $k$''.)
There are many smooth connected unipotent groups of dimension 1
over an imperfect field, and so this result makes precise the idea
that the class of commutative pseudo-reductive groups is big.
Corollary \ref{cordim1} also gives that
for every ordinary elliptic curve $A$ over 
a separably closed field $k$,
every $k$-wound group of dimension 1 occurs as the unipotent
quotient of a pseudo-abelian variety with abelian subvariety $A$.

On the other hand, we give some counterexamples. First, for
$k$ not separably closed, a $k$-wound group of dimension 1 need
not have any pseudo-reductive extension by $\Gm$ over $k$
(Example \ref{dim1ex}).
Conrad-Gabber-Prasad gave such an example in characteristic 3
\cite[equation 11.3.1]{CGP}, and we check the required property
in any characteristic at least 3.

Next, we exhibit a commutative $k$-wound group
of dimension 2 over a separably closed field $k$
which is not the unipotent quotient
of any commutative pseudo-reductive group (Example \ref{dim2}).
Finally, we exhibit a commutative $k$-wound group
over a separably closed field $k$ with $[k\colon k^p]=p$
which has no pseudo-reductive extension by $\Gm$ over $k$, although
it does have a pseudo-reductive extension by $(\Gm)^2$
(Example \ref{UU}). Question \ref{dimone} asks whether,
for a field $k$
with $[k\colon k^p]=p$, every commutative $k$-wound group
is the unipotent quotient of some pseudo-reductive group
over $k$. 

We now begin the proofs of these results. First we have a reduction
of the problem to the case of a separably closed field.

\begin{lemma}
\label{sep}
Let $U$ be a smooth connected commutative unipotent group
over a field $k$. Then $U$ is the unipotent quotient
of some commutative pseudo-reductive group over $k$
if and only if $U_{k_s}$ is the unipotent quotient of some
commutative pseudo-reductive group over the separable
closure $k_s$.
\end{lemma}

\begin{proof}
In one direction,
let $1\arrow T\arrow E\arrow U\arrow 1$ be a commutative pseudo-reductive
extension of $U$ by a torus $T$ over $k$. Then $E_{k_s}$ is
an extension $1\arrow T_{k_s}\arrow E_{k_s}\arrow U_{k_s}\arrow 1$,
and $E_{k_s}$ is pseudo-reductive, because the maximal smooth
connected affine normal $k_s$-subgroup
of $E_{k_s}$ is Galois-invariant and hence defined over $k$
\cite[Proposition 1.1.9]{CGP}.

Conversely, suppose that $U_{k_s}$ is the unipotent quotient
of some commutative pseudo-reductive group over $k_s$.
Then there is a finite separable extension $F$ of $k$
and an extension $1\arrow T\arrow E\arrow U_F\arrow 1$
of $U_F$ by a torus $T$ over $F$ such that $E$ is pseudo-reductive.
The Weil restriction $R_{F/k}E$ is an extension
$$1\arrow R_{F/k}T\arrow R_{F/k}E\arrow R_{F/k}(U_F)\arrow 1.$$
Here $R_{F/k}E$ is pseudo-reductive, by the universal
property of Weil restriction \cite[Proposition 1.1.10]{CGP}. Also,
$R_{F/k}T$ is a torus because $F$ is separable over $k$.
Finally, $U$ is a subgroup of $R_{F/k}(U_F)$ by the universal
property of Weil restriction. The inverse image of $U$
in $R_{F/k}E$ is a pseudo-reductive extension of $U$
by $R_{F/k}T$.
\end{proof}

We now begin to analyze extensions of unipotent groups
by the multiplicative group.
The group $\Ext^1(A,\Gm)$ of commutative extensions of an abelian variety
$A$ by the multiplicative group can be identified with the
group $\Pic^0(A)$ of isomorphism classes of
numerically trivial line bundles on $A$
\cite[Theorem VII.6]{SerreAlgebraic}.
For smooth connected commutative unipotent groups $U$,
it was known that $\Ext^1(U,\Gm)$ is a subgroup of $\Pic(U)$
\cite[Lemma 6.13.1]{KMT},
but the following lemma gives an explicit description of that
subgroup, analogous to what happens for abelian varieties.

\begin{lemma}
\label{ext}
Let $U$ be a smooth connected commutative unipotent group
over a field $k$. Then $\Ext^1(U,\Gm)$ is the subgroup of
elements $L\in\Pic(U)$ such that the translation $T_aL$ is isomorphic
to $L$ for all separable extension fields $F$ of $k$ (not necessarily
algebraic) and all $a\in U(F)$.
In short: $\Ext^1(U,\Gm)=\Pic(U)^U$.

The group $\Ext^1(U,\Gm)$ can also be described as the subgroup of primitive
elements in $\Pic(U)$, meaning that
$$\Ext^1(U,\Gm)=\{y\in \Pic(U): m^*(y)=\pi_1^*(y)+\pi_2^*(y)
\in Pic(U\times U)\},$$
where $m\colon U\times U\arrow U$ is the group operation
and $\pi_1,\pi_2\colon U\times U
\arrow U$ are the two projections.
\end{lemma}

\begin{proof}
Denote the group operation on $U$ by addition. We will use
Breen's spectral sequence for computing $\Ext$ groups
in the abelian category
of fppf sheaves over $k$ \cite{BreenBull}. One can also give
a more elementary but less efficient proof by imitating
Serre's proof of the analogous statement for abelian varieties
\cite[Theorem VII.5]{SerreAlgebraic}. 

For any commutative
$k$-group schemes $B$ and $C$, Breen's spectral sequence has the form
$$E_1^{i,j}=H^j_{\fppf}(X_i(B),C) \imp \Ext^{i+j}(B,C),$$
where the $k$-schemes $X_i(B)$ are explicit disjoint unions of powers of $B$,
starting with $X_0(B)=B$, $X_1(B)=B^2=B\times_k B$, and $X_2(B)=
B^3\coprod B^2$. The differential $d_1$ is an explicit alternating sum
of pullback maps. In particular, $d_1$ on the 0th column
is the homomorphism $\alpha\colon  H^j(B,C)\arrow H^j(B^2,C)$
given by $\alpha=m^*-\pi_1^*-\pi_2^*$.

We apply the spectral sequence to compute $\Ext^1(U,\Gm)$ for $U$ 
a smooth connected commutative unipotent group $U$
over a field $k$, with the following $E_1$ term. Since $\Gm$
is smooth over $k$, the fppf cohomology groups shown can also be viewed
as etale cohomology groups \cite[Theorem III.3.9]{Milne}.
$$\xymatrix{\cdots &&& \\
H^1(U,\Gm) \ar[r]\ar@{-->}[rrd]& H^1(U^2,\Gm) \ar[r]&
  \cdots & \\
H^0(U,\Gm) \ar[r]& H^0(U^2,\Gm) \ar[r]& H^0(U^3,\Gm) \oplus H^0(U^2,\Gm)
  \ar[r] & }$$
Since $U$ becomes isomorphic to affine space as a scheme
over the algebraic closure $\overline{k}$
\cite[Corollaire XVII.4.1.3]{SGA3},
we have $O(U^r)^*=k^*$
for every $r\geq 0$. It follows that the $d_1$ differential
on the zeroth row of the spectral sequence is exact, by comparing
with the spectral sequence computing $\Ext^*(0,\Gm)=0$.
In particular, the $d_2$ differential shown as a dotted arrow
maps into the zero group.
Therefore, the spectral sequence gives an isomorphism
$$\Ext^1(U,\Gm)=\ker(\alpha\colon  \Pic(U)\arrow \Pic(U\times U)),$$
as we want.
The right side is called the group of primitive line bundles on $U$.

We now prove the other description of $\Ext^1(U,\Gm)$.
For a primitive line bundle $L$ on $U$, fix a trivialization
of $L$ at the origin in $U$. Then there is an isomorphism
$m^*L\cong \pi_1^*L\otimes \pi_2^*L$, which is uniquely determined 
if we require it to be compatible with the trivialization of $L$
at $(0,0)$ in $U\times U$. (That isomorphism gives a canonical
isomorphism $L_{a+b}\cong L_a\otimes L_b$ for all $a,b\in U(F)$
and all extension fields $F$ of $k$.) Restricting that isomorphism
to $U$ times an $F$-rational point of $U$ gives
an isomorphism $T_aL\cong L$ on $U_F$ for all $a\in U(F)$,
and all extension
fields $F$ of $k$. Conversely, suppose that $T_aL\cong L$
for all $a\in U(F)$ and all separable extension fields $F$ of $k$.
We apply this to the function field $F=k(U)$ and $a\in U(F)$ the
generic point. Here $F$ is separable over $k$ since $U$ is smooth
over $k$. We can rewrite the isomorphism $T_aL\cong L$ on $U_F$
as $T_aL\cong L_a\otimes L$, since $L_a$ is just a 1-dimensional
$F$-vector space. This means that the line bundle
$M:=m^*(L)\otimes \pi_1^*(L^*)\otimes \pi_2^*(L^*)$ on $U\times U$
is trivial on $U\times (U-S)$ for some codimension-1 closed
subset $S$ of $U$. Therefore $M$ is linearly equivalent on $U\times U$ to
$\pi_2^*D$ for some divisor $D\subset U$ supported on $S$. Restricting
to $0\times U$, where $M$ is trivial, shows that $D$ is linearly equivalent
to 0 on $U$. So $M$ is trivial on $U\times U$. That is, $L$ is primitive.
\end{proof}

\begin{lemma}
\label{dim1pic}
Let $U$ be a $k$-wound group
of dimension 1 over a field $k$. Then $\Pic(U)\neq 0$.
\end{lemma}

Lemma \ref{dim1pic} was proved
by Kambayashi-Miyanishi-Takeuchi \cite[Theorem 6.5(i)]{KMT}.
We give a proof here for clarity.

\begin{proof}
Let $C$ be the unique regular compactification of $U$ over $k$.
Then $C-U$ is a single closed point, because $U$ becomes
isomorphic to $\A^1$ over the algebraic closure $\overline{k}$.
The group $\Pic(U)$ is the quotient of $\Pic(C)$ by the class
of the closed point $C-U$. I claim that the closed point $C-U$ has
degree a multiple of $p$ over $k$ (in fact, a power of $p$
greater than 1). It suffices to prove this
after passing to the separable closure $k_s$; then $U_{k_s}$ remains
$k$-wound and $C_{k_s}$ remains regular \cite[Prop.~X.6.5]{Bourbaki}. 
All finite field extensions of $k_s$ have degree a power of $p$,
so it suffices to show that $(C-U)(k_s)=\emptyset$. So suppose
that there is a $k_s$-rational point $w$ in $C-U$. Since $C$
is regular, it is smooth over $k_s$ near $w$. This gives a point
of $U(k_s((t)))$ that does not extend to $U(k_s[[t]])$, contradicting
a property of $k$-wound groups \cite[Proposition V.8]{Oesterle}.

Therefore, the degree homomorphism $\deg\colon \Pic(C)\arrow \Z$
passes to a well-defined homomorphism $\Pic(U)\arrow \Z/p$.
The homomorphism $\Pic(U)\arrow \Z/p$
is surjective, since the line bundle $O(0)$ on $U$
has degree 1, where $0\in U(k)$ is the identity element. 
\end{proof}

\begin{lemma}
\label{nonzero}
Let $U$ be a $k$-wound group
of dimension 1 over a separably closed field $k$.
Then $\Ext^1(U,\Gm)\neq 0$.
\end{lemma}

This can fail for $k$ not separably closed, by Example \ref{dim1ex}.

\begin{proof}
Let $C$ be the regular compactification of $U$ over $k$.
Let $\Pic_{C/k}$ be the Picard scheme \cite[Theorem 9.4.8]{Kleiman}.
Then $\Pic_{C/k}$ is a $k$-group scheme, locally of finite
type, with $\Pic(C_F)\cong \Pic_{C/k}(F)$ for every field
extension $F$ of $k$ (using that $H^0(C,O)=k$ and $C$ has a $k$-rational
point). Since $C$ is a geometrically irreducible projective curve,
the kernel $\Pic^0_{C/k}$
of the degree homomorphism $\Pic_{C/k}\arrow \Z$ is smooth, connected,
and of finite type over $k$
\cite[Theorem 8.2.3 and Proposition 8.4.2]{BLR}.
The curve $C$ becomes rational
over the algebraic closure $\overline{k}$, and so $\Pic^0_{C/k}$
is affine (as the abelian variety quotient of $(\Pic^0_{C/k})_{\overline{k}}$
is the Jacobian of the normalization of $C_{\overline{k}}$
\cite[section V.17]{SerreAlgebraic}, \cite[Proposition 9.2.10]{BLR}).
Because $U_{\overline{k}}$ is isomorphic
to $\A^1_{\overline{k}}$, the point $C_{\overline{k}}-U_{\overline{k}}$
corresponds to a single point on the normalization $\P^1_{\overline{k}}$
of $C_{\overline{k}}$, and so $(\Pic^0_{C/k})_{\overline{k}}$
is unipotent \cite[section V.17]{SerreAlgebraic},
\cite[Proposition 9.2.9]{BLR}.
It follows that $\Pic^0_{C/k}$
is unipotent.

The action of $U$ on itself by translation extends to an action
of $U$ on $C$, by the uniqueness of the regular
compactification $C$ and the smoothness of $U$. 
By the proof of Lemma \ref{dim1pic}, $\Pic(U)$ is an extension
of a finite cyclic group by the group $\Pic^0_{C/k}(k)$.
The action of $U(k)$ by translations on $\Pic(U)$ clearly restricts
to the action of $U(k)$ on $\Pic^0_{C/k}(k)$ by translations.

If $\Pic^0_{C/k}$ is zero, then $\Pic_{C/k}$ is isomorphic to $\Z$
by the degree. Then the action of $U$ on $\Pic_{C/k}$ is trivial,
since $U$ is connected. In this case,
$\Pic(U_F)$ is a finite cyclic
group for all separable extension fields $F$ of $k$, and $U(F)$ acts
trivially on $\Pic(U_F)$ since $\Pic(C_F)\arrow \Pic(U_F)$
is surjective. So $\Ext^1(U,\Gm)=\Pic(U)$ in this case
(using Lemma
\ref{ext}) and this is a nonzero cyclic group
by Lemma \ref{dim1pic}. (For this case, we did not need
$k$ to be separably closed.)

Otherwise, $\Pic^0_{C/k}$ is not zero. In this case, we will
show that the subgroup $\Pic^0(C)^U$ of $\Ext^1(U,\Gm)$ is not zero,
using the notation of Lemma \ref{ext}.
Since $P:=\Pic^0_{C/k}$ is
a smooth connected commutative unipotent $k$-group,
the semidirect product
$U\ltimes P$ is unipotent, and therefore is a nilpotent group
by the results listed in section \ref{notation}. That
implies that the action of $U$ on $P$ must be nilpotent.
In more detail, write $(u-1)q$ to mean $uq-q$ for any extension field
$F$ of $k$, $u\in U(F)$, and
$q\in P(F)$, where the group operation on $P$ is written additively.
If we define $P^m$ for each natural number $m$
as the closed subgroup of $P$
generated by elements $(u_1-1)\cdots (u_m-1)q$ for $u_i\in U(k_s)$
and $q\in P(k_s)$, then the subgroups $P=P^0\supset P^1 \supset
P^2\supset\cdots $ are closed and connected, eventually equal to zero
because the group $U\ltimes P$ is nilpotent. Also,
the group $U$ acts trivially on each $P^m/P^{m+1}$.

In particular, the last $P^m$ not equal to zero is a nontrivial
smooth connected subgroup of $\Pic^0_{C/k}$ such that $U$
acts trivially on $P^m$. Thus $P^m(k)\subset \Ext^1(U,\Gm)$
by Lemma \ref{ext}. Since $k$ is separably closed,
$P^m(k)\neq 0$.
\end{proof}

\begin{corollary}
\label{cordim1}
Let $U$ be a $k$-wound
group of dimension 1 over a field $k$. Then $U$ is the unipotent
quotient of some commutative pseudo-reductive group $E$ over $k$.

Suppose in addition that $k$ is separably closed. Then we can take
$E$ to be an extension of $U$ by $\Gm$.
Also, for any ordinary
elliptic curve $A$ over $k$, there is an extension of $U$
by $A$ which is a pseudo-abelian variety.
\end{corollary}

Recall that Corollaries \ref{supersingular} and \ref{ordinary}
give a larger
class of pseudo-abelian varieties when
the abelian subvariety is a supersingular elliptic curve,
or an ordinary elliptic curve which cannot be defined over the subfield
$k^p$.

\begin{proof}
By Lemma \ref{sep}, we can assume that $k$ is separably closed.
By Lemma \ref{nonzero}, there is a nontrivial
extension
$$1\arrow \Gm\arrow E\arrow U\arrow 1$$
of commutative $k$-groups. If $N$ is a nontrivial
smooth connected unipotent $k$-subgroup of $E$,
then $N\cap \Gm=1$ as a group scheme, and so $N$
projects isomorphically to a subgroup of $U$. Since
$U$ has dimension 1, $N$ projects isomorphically to $U$,
contradicting that the extension is nontrivial. So $E$
must be pseudo-reductive.

An ordinary elliptic curve $A$ over $k$ has $\ker(F)$ isomorphic
to $\mu_p$, since $k$ is separably closed. The existence
of the pseudo-reductive extension $E$ implies that there
is a pseudo-abelian extension of $U$ by $A$,
by Corollary \ref{ordinary} and Lemma \ref{equiv}.
\end{proof}

\begin{example}
\label{dim1ex}
Let $k_0$ be a field of characteristic $p\geq 3$ and let
$k$ be the rational function field $k_0(t)$.
Let $U$ be the subgroup $\{(x,y):y^p=x-tx^p\}$ of $(\Ga)^2$ over $k$.
Then $U$ is a $k$-wound group of dimension 1
with $\Ext^1(U,\Gm)=0$. Therefore, $U$ has no extension by $\Gm$
over $k$ which is pseudo-reductive.
\end{example}

Conrad-Gabber-Prasad observed that $\Ext^1(U,\Gm)=0$ in this example 
when $p=3$ \cite[equation 11.3.1]{CGP}. Note that $U$
does have an pseudo-reductive extension by some torus over $k$,
by Corollary \ref{cordim1}.

\begin{proof}
Over $\overline{k}$, $U$ becomes isomorphic to $\Ga$ by a simple change
of variables. So $U$
is connected and smooth over $k$.
If $U$ were isomorphic to $\Ga$, then the projective
closure $X=\{ [x,y,z]\in \P^2: y^p=xz^{p-1}-tx^p\}$
of $U$ would have normalization isomorphic to $\P^1$
over $k$, and the image of $\infty\in \P^1$ would be a $k$-rational
point in $X-U$. But there is no such point, and so $U$ is $k$-wound.

By Kambayashi-Miyanishi-Takeuchi \cite[6.13.3]{KMT},
$$\Ext^1(U,\Gm)\cong \bigg\{ (c_0,\ldots,c_{p-2})\in k^{p-1}:
c_{p-2}=\sum_{0\leq j\leq p-2} c_j^pt^j \bigg\}.$$
We will show that this equation has no nonzero solutions
in $k=k_0(t)$. We can assume that $k_0$ is algebraically closed.

Suppose that $(c_0,\ldots,c_{p-2})$ is a nonzero element
of $\Ext^1(U,\Gm)$. If $c_{p-2}=0$, then the equation
gives that $1,t,\ldots,t^{p-3}$ are linearly dependent
over the field $k^p$, which is false. So $c_{p-2}\neq 0$.

Viewing $c_0,\ldots,c_{p-2}$ as rational functions over $k_0$,
we can differentiate the equation $p-2$ times to get
$c_{p-2}^{(p-2)}=(p-2)!\, c_{p-2}^p$. By considering the pole
order of $c_{p-2}$ at each point $a\in k_0$, we deduce from this
equation that $c_{p-2}$ is regular at each point $a\in k_0$.
Since $k_0$ is algebraically closed, that means that
$c_{p-2}$ is a polynomial over $k_0$. Let $d$ be its
degree. Then $c_{p-2}^p$ is nonzero of degree $pd$ while
$c_{p-2}^{(p-2)}$ has lower degree, a contradiction.
We have shown that $\Ext^1(U,\Gm)=0$ over $k=k_0(t)$.
\end{proof}

\begin{example}
\label{dim2}
Let $k_0$ be a field of characteristic $p>0$, and let
$k$ be the rational function field $k_0(a,b)$. Let $U$ be the subgroup
$$\{(x,y,z):x+ax^p+by^p+z^p=0\}$$
of $(\Ga)^3$. Then $U$ is a commutative $k$-wound group
of dimension 2 with $\Ext^1(U_{k_s},\Gm)=0$. It follows
that, even over the separable closure $k_s$,
$U$ is not the unipotent quotient of any commutative
pseudo-reductive group.
\end{example}

\begin{proof}
Let $X$ be the projective closure of $U$,
$$X=\{ [x,y,z,w]\in \P^3_k: xw^{p-1}+ax^p+by^p+z^p=0\}.$$
Then $X$ has no $k$-points at infinity (meaning points with $w=0$).
It follows that $U$ is $k$-wound.

Since $\Ext^1(U_{k_s},\Gm)$ is a subgroup of $\Pic(U_{k_s})$
(Lemma \ref{ext}), it suffices to show that $\Pic(U_{k_s})=0$.
We start by finding the non-regular locus of the surface $X$.
To do so, we compute the zero locus
of all derivatives of the equation with respect to $x,y,z,w$
and also $a,b$: this gives that $w^{p-1}=0$, $x^p=0$,
$y^p=0$, and hence $z^p=0$, which defines the empty set
in $\P^3_k$. So $X$ is regular, and it follows that $X_{k_s}$
is regular \cite[Prop.~X.6.5]{Bourbaki}. 
Also, $X-U$ is the plane curve $D=\{[x,y,z]
\in \P^2_k: ax^p+by^p+z^p=0\}$, which is regular over $k_s$
and hence irreducible over $k_s$. It follows that 
$$\Pic(U_{k_s})\cong \Pic(X_{k_s})/\Z \cdot [D_{k_s}]=\Pic(X_{k_s})/
\Z\cdot O(1).$$
So it suffices to show that $\Pic(X_{k_s})=\Z\cdot O(1)$.

\begin{lemma}
\label{inj}
Let $Y$ be a scheme of finite type over a field $F$ such that
$H^0(Y,O)=F$. Then the homomorphism $\Pic(Y)\arrow \Pic(Y_E)$ is injective
for any extension field $E$ of $F$.
\end{lemma}

\begin{proof}
We have $H^0(Y_E,O)=H^0(Y,O)\otimes_F E=E$.
Let $L$ be a line bundle on $Y$ which becomes trivial over $E$.
Then $L$ and the dual line bundle $L^*$ have 1-dimensional spaces
of sections over $Y$, since that is
true over $Y_E$. Let $s\in H^0(Y,L)$ and $t\in H^0(Y,L^*)$ be nonzero
sections. Then the product $st\in H^0(Y,O)=F$ is not zero since
that is true over $E$. This means that the compositions
$O_Y\stackrel{s}{\longrightarrow} L \stackrel{t}{\longrightarrow} O_Y$
and $L\stackrel{t}{\longrightarrow} O_Y \stackrel{s}{\longrightarrow} L$
are isomorphisms. So $L$ is trivial.
\end{proof}

A referee pointed out that one can prove Lemma \ref{inj} under the weaker
assumption that the ring $O(Y)$ has trivial Picard group. Consider
the morphism $f\colon Y\arrow S:=\Spec \, O(Y)$. Then the Leray spectral
sequence for fppf cohomology gives (since $f_*\Gm=\Gm$)
that $\Pic(Y)/\Pic(S)$ injects into $H^0_{\fppf}(S,R^1f_*\Gm)$,
which gives the result.

Since $X$ is a surface in $\P^3$, we have $H^0(X_{k_s},O)=k_s$
by the exact sequence of sheaves $0\arrow O_{\P^3}(-X)\arrow O_{\P^3}
\arrow O_X\arrow 0$. By Lemma \ref{inj}, 
we have $\Pic(X_{k_s})
=\Z\cdot O(1)$ as we want if we can show that $\Pic(X_{\overline{k}})=\Z\cdot
O(1)$. We have 
\begin{align*}
X_{\overline{k}}&\cong \{[x,y,z,w]\in \P^3: xw^{p-1}+x^p+y^p+z^p=0\}\\
&\cong \{[x,y,z,w]\in \P^3: xw^{p-1}+y^p=0\}
\end{align*}
Thus $X_{\overline{k}}$ is the projective cone over the plane curve
$xw^{p-1}+y^p=0$. 

Let $Y$ be a projective scheme over a field $k$
such that $H^0(Y,O_Y)=k$, and let $O_Y(1)$ be an ample line
bundle on $Y$. Let $R$ be the homogeneous coordinate
ring $\oplus_{j\geq 0} H^0(Y,O_Y(j))$ as a graded ring, and
define the {\it projective cone }over $Y$ to be
$X=\Proj R[x]$, where $x$ has degree 1.
For a closed subscheme $Y\subset \P^n$ over $k$, there is a finite
morphism from $X$ to the {\it classical projective cone }over $Y$
in $\P^{n+1}$, which is an isomorphism away from the vertex
\cite[section 2.56]{Kollarsing}. This morphism is an isomorphism
if the $k$-algebra $\oplus_{j\geq 0}H^0(Y,O_Y(j))$ is generated
by $H^0(\P^n,O(1))$, but in general the projective cone
as defined here has better properties.

\begin{lemma}
\label{cone}
Let $Y$ be a projective scheme over a field $k$ such that
$H^0(Y,O_Y)=k$, and let $O_Y(1)$ be an ample
line bundle on $Y$. Let $X$ be
the projective cone over $Y$.
Then $\Pic(X)=\Z\cdot O_X(1)$.
\end{lemma}

\begin{proof}
Let $Z$ be the $\P^1$-bundle $P(O_Y\oplus O_Y(1))$ over $Y$.
By the calculation of the $K$-theory of projective bundles
\cite[Theorem VI.1.1]{SGA6},
$\Pic(Z)\cong \Pic(Y)\oplus \Z$ for any connected
scheme $Y$.
(Here the summand $\Z$ is generated by the natural line bundle
$O_Z(1)$ on the projective bundle $Z$.
The statement means that every line bundle on $Z$ is, in a unique
way, a pullback from $Y$
tensored with $O_Z(j)$ for some integer $j$.)
Since $H^0(Y,O_Y)=k$, $Y$ is connected and so 
$\Pic(Z)=\Pic(Y)\oplus \Z$.
Since $Y$ is projective over $k$ with $H^0(Y,O_Y)=k$,
there is a surjection $f\colon Z\arrow X$ which contracts a copy
of $Y$ (the section corresponding to the first projection
$O_Y\oplus O_Y(1)\surj O_Y$ over $Y$)
to a point \cite[Proposition 8.6.2]{EGAII}.
For any line bundle $L$ on $X$, the pullback
$f^*L$ is trivial on $Y$, and so the image
of $f^*\colon \Pic(X)\arrow \Pic(Z)$
is contained in $\Z\cdot O_Z(1)$. (By restricting to a fiber
of the $\P^1$-bundle $Z\arrow Y$,
we see that $f^*O_X(1)\cong O_Z(1)$.)
It remains to show
that $f^*\colon \Pic(X)\arrow \Pic(Z)$ is injective.

The natural map $O_X\arrow f_*O_Z$ is an isomorphism
\cite[Proposition 8.8.6]{EGAII}.
So, for any line bundle
$L$ on $X$, the natural map $L\arrow f_*f^*(L)$ is an isomorphism.
If $L$ is a line bundle on $X$ whose pullback to $Z$ is trivial,
then $H^0(X,L)= H^0(X,f_*f^*L)\cong H^0(Z,f^*(L))\cong H^0(Z,O_Z)=k$.
Likewise, $H^0(X,L^*)\cong k$. It follows that $L$ is trivial,
as in the proof of Lemma \ref{inj}.
\end{proof}

We now return to Example \ref{dim2}. The surface $X_{\overline{k}}$
is the classical projective cone over the plane curve
$Y=\{ xw^{p-1}+y^p=0\}$ over $\overline{k}$. For a curve
$Y$ of any degree $d$ in $\P^2$,
the $k$-algebra
$\oplus_{j\geq 0}H^0(Y,O_Y(j))$ is generated in degree 1,
by considering the exact sequence of sheaves on $\P^2$,
$0\arrow O_{\P^2}(j-d)\arrow O_{\P^2}(j)\arrow O_Y(j)\arrow 0$.
So $X_{\overline{k}}$ is the projective cone over $Y$ in the sense
defined above.
By Lemma \ref{cone}, $\Pic(X_{\overline{k}})=\Z\cdot O_X(1)$.
As we have said, it follows that $\Pic(U_{k_s})=0$. Example
\ref{dim2} is proved.
\end{proof}

The following example, supplied by a referee, answers a question
in the original version of this paper. Note that $k$ can be
separably closed in the following example.

\begin{example}
\label{UU}
Let $k$ be a field of characteristic $p>0$
with $[k\colon k^p]=p$. Let $k_1=k^{1/p}$, which is
an extension of degree $p$ of $k$.
Let $U$ be the smooth connected
commutative unipotent $k$-group $(R_{k_1/k}\Gm)/\Gm$ of dimension $p-1$.
Then $U\times U$ is $k$-wound, but
$U\times U$ has no extension by $\Gm$ over $k$ which is pseudo-reductive.
It does have an extension by $(\Gm)^2$ which is pseudo-reductive.
\end{example}

\begin{proof}
We first consider a more general situation. Let $k$ be any field
of characteristic $p>0$.
For any smooth
connected commutative affine $k$-group $G$ with maximal
torus $T$, let $K$ be the field of definition over $k$ of the geometric
unipotent radical of $G$. Thus $G_K$ is the product of $T_K$
with a smooth connected unipotent $K$-group; in particular,
we have a unique splitting $G_K\arrow T_K$ of the inclusion.
By the universal property of Weil restriction, this gives
a homomorphism $f\colon G\arrow R_{K/k}(T_K)$ which restricts to the obvious
inclusion $T\arrow R_{K/k}(T_K)$. Moreover, $f$
does not factorize through $R_{L/k}(T_L)$ for any 
proper subextension $L$ of $K/k$. Let $k_1$ denote the extension
field $k^{1/p}$.
Suppose that $p$ kills $G/T$; then the image in $T(K)$
of a point in $G(k)$ has $p$th power
in $T(k)$, and so that image lies in $T(k_1)$. It follows that 
$f$ factors through $R_{L/k}(T_L)$ for $L = K\cap k_1$,
and so $K$ is contained in $k_1$.

We now return to the notation of this Example, so that
$k$ is a field with $[k\colon k^p]=p$. Then $k_1$ is equal to $k(t^{1/p})$
for any element $t\in k^*$ which is not a $p$th power.
We know that $U$
is $k$-wound, because $R_{k_1/k}\Gm$ is pseudo-reductive.
(Use the universal property of Weil restriction: a homomorphism
$\Ga\arrow R_{k_1/k}\Gm$ over $k$
is equivalent to a homomorphism $\Ga\arrow \Gm$ over $k_1$,
which must be trivial.) The product
$(R_{k_1/k}\Gm)^2$ is an extension of $U\times U$
by $(\Gm)^2$ which is pseudo-reductive.

By Oesterl\'e, as we used in the proof of Lemma
\ref{nontrivial}, $U$
is isomorphic to
$$\{(x_0,\ldots,x_{p-1})\in (\Ga)^p: x_0^p+tx_1^p+\cdots
+t^{p-1}x_{p-1}^p=x_{p-1}\}$$
\cite[Proposition VI.5.3]{Oesterle}. The Lie algebra of $U$ is
a restricted Lie algebra with $p$th power operation equal to zero
(since that is true for $(\Ga)^p$, for example).
So every nonzero element of the Lie algebra of $U$ gives an
$\alpha_p$ subgroup of $U$. The intersections of $U$ with
one-dimensional $k$-linear subspaces of $(\Ga)^p$ give exactly
the $k$-subgroup schemes of order $p$ in $U$. Therefore
the quotient of $U$ by any $k$-subgroup scheme of order $p$
(in particular, any $\alpha_p$ subgroup) is isomorphic
to $(\Ga)^{p-1}$ over $k$. So any
homomorphism $U\arrow U$ over $k$ is either zero or induces
an isomorphism on Lie algebras, using that $U$ is $k$-wound.

Now let $G$ be a commutative extension of $U\times U$ by $\Gm$
over $k$. Since $U\times U$ is killed by $p$, we showed above
that the geometric unipotent radical of $G$ is defined
over $k_1$. As above, this gives a homomorphism
$G\arrow R_{k_1/k}\Gm$ which is the identity on the subgroup
$\Gm$. On the quotients by $\Gm$, this gives
a homomorphism $h\colon U\times U\arrow U$, and the extension $G$
of $U\times U$ by $\Gm$ is pulled back via $h$. By the previous paragraph,
either $h$ is zero or $h$ induces a surjection on Lie algebras.
In both cases, $\ker(h)$ is a smooth $k$-subgroup of positive
dimension. Since the extension $G$ of $U\times U$ by $\Gm$ splits
over $\ker(h)$, $G$ is not pseudo-reductive.
\end{proof}

The following question is suggested by Corollary \ref{cordim1},
Example \ref{dim2}, and Example \ref{UU}.

\begin{question}
\label{dimone}
If $k$ is a field
with $[k\colon k^p]=p$, is every $k$-wound commutative unipotent group
the unipotent quotient of some commutative pseudo-reductive group
over $k$?
\end{question}

In view of Example \ref{UU}, the maximal torus
of the pseudo-reductive group will in general have dimension greater than 1.
By Lemma \ref{sep}, it suffices to answer Question \ref{dimone}
for $k$ separably closed.

We know that the unipotent quotient of a commutative pseudo-reductive
group is $k$-wound. So Question \ref{dimone} would describe exactly
which groups occur as the unipotent quotients of commutative
pseudo-reductive groups over a field $k$
with $[k\colon k^p]=p$. (For example, that would apply to
the function field of a curve over a finite field.)


\small \sc DPMMS, Wilberforce Road,
Cambridge CB3 0WB, England

b.totaro@dpmms.cam.ac.uk

\begin{thebibliography}{99}
\bibitem{SGA6} P.~Berthelot, A.~Grothendieck, and L.~Illusie.
{\it Th\'eorie des intersections et th\'eor\`eme de Riemann-Roch
(SGA 6). }Lecture Notes in Mathematics 225, Springer (1971).

\bibitem{Borel} A.~Borel. {\it Linear algebraic
groups. }Springer (1991).

\bibitem{BLR} S.~Bosch, W.~L\"utkebohmert, and M.~Raynaud.
{\it N\'eron models. }Springer (1990).

\bibitem{Bourbaki} N.~Bourbaki. {\it Alg\`ebre commutative.
Chapitre 10. }Springer (2007).

\bibitem{BreenBull} L.~Breen. On a nontrivial higher extension
of representable abelian sheaves. {\it Bull.\ Amer.\
Math.\ Soc.\ }{\bf 75 }(1969), 1249--1253.

\bibitem{BreenENS} L.~Breen. Un th\'eor\`eme d'annulation
pour certains $\Ext^i$ de faisceaux ab\'eliens.
{\it Ann.\ Sci.\ Ecole Normale Sup\'erieure }{\bf 5 }(1975), 339--352.

\bibitem{Brion} M.~Brion. Anti-affine algebraic groups.
{\it J.\ Alg.\ }{\bf 321 }(2009), 934--952.

\bibitem{Chevalley} C.~Chevalley. Une d\'emonstration d'un
th\'eor\`eme sur les groupes alg\'ebriques.
{\it J.\ Math.\ Pure Appl.\ }{\bf 39 }(1960), 307--317.

\bibitem{Conradchev} B.~Conrad. A modern proof of Chevalley's theorem
on algebraic groups. {\it J.\ Ramanujan Math.\ Soc.\ }{\bf 17 }(2002),
1--18.

\bibitem{Conradtrace} B.~Conrad. Chow's $K/k$-image and $K/k$-trace,
and the Lang-N\'eron theorem. {\it Enseign.\ Math.\ }{\bf 52 }(2006),
37--108.

\bibitem{CGP} B.~Conrad, O.~Gabber, and G.~Prasad. {\it Pseudo-reductive
groups. }Cambridge (2010).

\bibitem{DG} M.~Demazure and P.~Gabriel. {\it Groupes
alg\'ebriques. }Masson (1970).

\bibitem{SGA3} M.~Demazure and A.~Grothendieck. {\it Sch\'emas
en groupes }I, II, III (SGA 3). Lecture Notes in Mathematics 151,
152, 153, Springer (1970). Revised version edited by P.~Gille
and P.~Polo, vols.\ I and III, Soc.\ Math.\ de France (2011).

\bibitem{EGAII} A.~Grothendieck. El\'ements de g\'eom\'etrie
alg\'ebrique. II. Etude globale \'el\'ementaire de quelques classes
de morphismes. {\it Publ.\ Math. IHES }{\bf 8 }(1961), 5--222.

\bibitem{EGAIII2} A.~Grothendieck. El\'ements de g\'eom\'etrie
alg\'ebrique. III. Etude cohomologique des faisceaux coh\'erents.
Seconde partie. {\it Publ.\ Math. IHES }{\bf 17 }(1963), 5--91.

\bibitem{SGA7} A.~Grothendieck. {\it Groupes de monodromie
en g\'eom\'etrie alg\'ebrique }I (SGA 7). 
Lecture Notes in Mathematics 288, Springer (1972).

\bibitem{Hartshorne} R.~Hartshorne. {\it Algebraic
geometry. }Springer (1977).

\bibitem{KMT} T.~Kambayashi, M.~Miyanishi, and M.~Takeuchi.
{\it Unipotent algebraic groups. }Lecture Notes in Mathematics
414, Springer (1974).

\bibitem{Kato} K.~Kato. A generalization of local class field
theory by using $K$-groups. II. {\it J.\ Fac.\ Sci. Univ.\
Tokyo Sect.\ 1A Math.\ }{\bf 27 }(1980), 603--683.

\bibitem{KM} N.~Katz and B.~Mazur. {\it Arithmetic moduli
of elliptic curves. }Princeton (1985).

\bibitem{Kleiman} S.~Kleiman. The Picard scheme.
{\it Fundamental algebraic geometry, }235--321.
Amer.\ Math.\ Soc. (2005).

\bibitem{Kollar} J.~Koll\'ar. {\it Rational curves
on algebraic varieties. }Springer (1996).

\bibitem{Kollarsing} J.~Koll\'ar. {\it Singularities
of the minimal model program. }Cambridge (2013).

\bibitem{Lang} S.~Lang. {\it Fundamentals of Diophantine
geometry. }Springer (1983).

\bibitem{Milne} J.~Milne. {\it Etale cohomology. }Princeton (1980).

\bibitem{Oesterle} J.~Oesterl\'e. Nombres de Tamagawa et groupes
unipotents en caract\'eristique $p$. {\it Invent.\ Math.\ }{\bf 78 }(1984),
13--88.

\bibitem{Raynaud} M.~Raynaud. {\it Faisceaux amples
sur les sch\'emas en groupes et les espaces
homog\`enes. }Lecture Notes in Mathematics 119, Springer (1970).

\bibitem{Rosenlichtbasic} M.~Rosenlicht. Some basic theorems
on algebraic groups. {\it Amer.\ J.\ Math.\ }{\bf 78 }(1956),
401--443.

\bibitem{Rosenlichtsome} M.~Rosenlicht. Some rationality questions
on algebraic groups. {\it Ann.\ Mat.\ Pura Appl.\ }(4) {\bf 43 }(1957),
25--50.

\bibitem{SS} C.~Sancho de Salas and F.~Sancho de Salas. Principal
bundles, quasi-abelian varieties and structure of algebraic groups.
{\it J.\ Alg.\ }{\bf 322 }(2009), 2751--2772.

\bibitem{Schroer} S.~Schr\"oer. On genus change in algebraic
curves over imperfect fields. {\it Proc.\ Amer.\ Math.\
Soc.\ }{\bf 137 }(2009), 1239--1243.

\bibitem{SerreAlgebraic} J.-P.~Serre. {\it Algebraic groups
and class fields. }Springer (1988).

\bibitem{Stohr} K.-O.~St\"ohr. Gorenstein hyperelliptic curves.
{\it J.\ Pure Appl.\ Algebra }{\bf 135 }(1999), 93--105.

\bibitem{Tategenus} J.~Tate. Genus change in inseparable extensions
of function fields. {\it Proc.\ Amer.\ Math.\ Soc.\ }{\bf 3 }(1952),
400--406.

\bibitem{TateHonda} J.~Tate. Classes d'isog\'enie des vari\'et\'es
ab\'eliennes sur un corps fini (d'apr\`es T.~Honda).
{\it S\'eminaire Bourbaki, }vol.\ 1968/69, Expos\'es 347--363,
Springer (1971), 95--110.

\bibitem{Tits91} J.~Tits. Th\'eorie des groupes. {\it Annuaire du
Coll\`ege de
France, }1991--1992,
115--133.

\bibitem{Tits92} J.~Tits. Th\'eorie des groupes. {\it Annuaire du
Coll\`ege de
France, }1992--1993,
113--131.

\bibitem{Witt} E.~Witt. $p$-Algebren und Pfaffsche Formen.
{\it Abh.\ Math.\ Sem.\ Univ.\ Hamburg }{\bf 22 }(1958),
308--315. Also in E.~Witt,
{\it Collected papers. Gesammelte Abhandlungen}, Springer (1998).

\end{thebibliography}
\end{document}